\documentclass[11pt,leqno]{amsart}
\usepackage[dvipsnames]{color}
\usepackage{amsfonts,amssymb,amsmath,amscd,amstext}
\numberwithin{equation}{section}
\usepackage{amsthm}
\usepackage[colorlinks=true,linkcolor=teal,citecolor=purple]{hyperref}
\usepackage[utf8]{inputenc}
\usepackage{graphicx}
\usepackage{changes}
\usepackage{comment}
\usepackage{mathtools}
\usepackage{tikz-cd}
\usepackage[noabbrev,capitalize,nameinlink]{cleveref}
\usepackage{mathdots}
\usepackage[a4paper,top=2.1cm,bottom=2.1cm,left=1.6cm,right=1.6cm]{geometry}

\usepackage{ulem}
\usepackage{soul}

\renewcommand{\leq}{\leqslant}
\renewcommand{\le}{\leqslant}

\newcommand{\Rn}{\mathbb{R}^n}

\newcommand{\X}{X}
\newcommand{\rr}{{\bf R}}

\newcommand{\classuno}{\mathcal S_1}
\newcommand{\classunoraff}{\mathcal{T}_1}

\newcommand{\classdue}{\mathcal S_2}

\newcommand{\classdueraff}{\mathcal T_2}

\newcommand{\D}{\mathcal D}

\newcommand{\Om}{\Omega}

\newcommand{\mA}{\mathcal{A}}

\newcommand{\scu}{\longrightarrow}

\newcommand{\lp}[1]{L^{p}(#1)}

\newcommand{\anso}[1]{W^{1,p}_X(#1)}

\newcommand{\Cm}{\C^{\dagger}}

\newcommand{\C}{\mathcal{C}}

\newcommand{\les}{<}
\newcommand{\ges}{>}

\newcommand{\clam}{\mathcal{M}_{p}(e_1,e_2,e_3,e_4,\mu)}

\newcommand{\average}{{\mathchoice {\kern1ex\vcenter{\hrule height.4pt
				width 6pt
				depth0pt} \kern-9.7pt} {\kern1ex\vcenter{\hrule height.4pt width 4.3pt
				depth0pt}
			\kern-7pt} {} {} }}
\newcommand{\ave}{\average\int}

\usepackage{amssymb,epsfig,xcolor,upgreek, cancel,enumitem}

\newcommand{\R}{{\bf R}}
\newcommand{\N}{{\bf N}}

\newfont{\bbten}{bbold12}

\definechangesauthor[color=red]{J}
\definecolor{champagne}{rgb}{0.97, 0.91, 0.81}

\definecolor{asparagus}{rgb}{0.53, 0.66, 0.42}

\DeclareMathOperator{\divv}{div}

\DeclareMathOperator{\diver}{div}

\DeclareMathOperator{\supp}{supp}

\DeclareMathOperator{\Lip}{Lip}

\newtheorem{theorem}{Theorem}[section]
\newtheorem{proposition}[theorem]{Proposition}
\newtheorem{definition}[theorem]{Definition}
\newtheorem{lemma}[theorem]{Lemma}

\theoremstyle{definition}
\newtheorem{example}[theorem]{Example}
\newtheorem{remark}[theorem]{Remark}
\definechangesauthor[name=Simone, color=blue]{S}

\title[Variational convergences under moving anisotropies]{Variational convergences under moving anisotropies}

\author[A.\ Maione]{Alberto Maione}
\author[F.\ Paronetto]{Fabio Paronetto}
\author[S.\ Verzellesi]{Simone Verzellesi}

\address[A.\ Maione]{ Dipartimento di Matematica, Politecnico di Milano, \newline \indent Campus Leonardo, Edificio 14 "Nave", Via Edoardo Bonardi 9, 20133 Milano - Italy}
\email{alberto.maione@polimi.it}

\address[F.\ Paronetto]{ Dipartimento di Matematica ``Tullio Levi-Civita'', Universit\`a di Padova, \newline \indent via Trieste 63 35121, Padova - Italy}
\email{fabio.paronetto@unipd.it}

\address[S.\ Verzellesi]{Department of Decision Sciences and BIDSA, Bocconi University \protect\newline
\indent Via R\"ontgen 1, 20136, Milano - Italy}
\email{simone.verzellesi@unibocconi.it}

\thanks{{\it Acknowledgements.} The authors thank Fares Essebei, Mattia Freguglia, Andrea Pinamonti, Francesco Serra Cassano and Giacomo Vianello for interesting and valuable conversations on the topic of the paper.}
\begin{document}
\begin{abstract}
We study the asymptotic behaviour of sequences of integral functionals depending on moving anisotropies. 
We introduce and describe the relevant functional setting, establishing uniform Meyers-Serrin type approximations, Poincaré inequalities and compactness properties. 
We prove several $\Gamma$-convergence results, and apply the latter to the study of $H$-convergence of anisotropic linear differential operators.
\end{abstract}
\maketitle


\section{Introduction}
In this paper we study the asymptotic behaviour of sequences of functionals of the form
\[
\int_\Om f_h(x,X^hu(x))\,{\rm d}x,\quad u\in W^{1,p}_{X^h}(\Om),
\]
for a sequence of anisotropies $(X^h)_h= (X_1^h, \ldots, X_m^h)_h$, where $\Omega$ is a bounded open subset of $\R^n$.
As $h$ tends to $\infty$, we describe the convergence of these functionals, their minima, minimizers and momenta, to the corresponding limit quantities, associated with a limit functional
\begin{equation}\label{introlimitfun}
    \int_\Om f(x,Xu(x))\,{\rm d}x,\quad u\in W^{1,p}_{X}(\Om),
\end{equation}
driven by a limit anisotropy $X=(X_1,\ldots,X_m)$. 
The natural variational convergence to take into account is \emph{$\Gamma$-convergence} (see e.g. \cite{MR1968440,MR1201152}).
This technique, which is extremely useful both for theoretical purposes and practical applications, remains a crucial tool in the study of various problems in mechanics, such as homogenization, phase transitions and the asymptotic analysis of PDEs.
In addition to an interest in $\Gamma$-convergence itself, we explore its applications to the study of anisotropic, linear, symmetric PDEs, referring back to the so-called \emph{$G$-convergence} or \emph{$H$-convergence} (see e.g. \cite{pankovbook,tartarbook}
).

\medskip
The investigation of anisotropic variational functionals as in \eqref{introlimitfun}, as well as their related functional frameworks, is originally motivated by L. H\"ormander's seminal work on \emph{hypoelliptic operators} \cite{Hörmander1967147}. 
The latter constitutes a milestone in the study of differential operators with underlying \emph{sub-Riemannian} type structures. 
Important evidences are the works of G.B. Folland \cite{Folland1975161} and of L. Rothschild and E.M. Stein \cite{MR0436223} in the context of \emph{stratified} and \emph{nilpotent Lie groups}. We refer to \cite{MR2363343} for further accounts on analysis on Lie groups, and to \cite{MR3971262,MR1421823,Enrico2025book} for thorough introductions to sub-Riemannian geometry.

\medskip
As a prototypical example, consider the smooth anisotropy $X=(X_1,X_2)$ on $\R^3$, where 
\begin{equation}\label{anisotropiheisenberg}
    X_1=\frac{\partial}{\partial x_1}+x_2\frac{\partial}{\partial x_3}\qquad\text{and}\qquad X_2=\frac{\partial}{\partial x_2}-x_1\frac{\partial}{\partial x_3}.
\end{equation}
This generates, via Lie-brackets, the Lie algebra of the \emph{first Heisenberg group}, and induce on it a \emph{sub-Riemannian structure}. 
The relevant hypoelliptic operator associated with $X$ is the so-called \emph{sub-Laplacian}
\begin{equation}\label{laplaciansub}
    \left(X_1\right)^2u+\left(X_2\right)^2 u,
\end{equation}
and its associated \emph{Dirichlet energy} reads as
\begin{equation} \label{dirienesub}
    \int_\Omega \left(|X_1u|^2+|X_2 u|^2\right)\,{\rm d}x.
\end{equation}
As \eqref{laplaciansub} is not elliptic, \eqref{dirienesub} is not coercive. 
Moreover, the Euclidean Sobolev spaces $W^{1,2}(\Om)$ are not the correct finiteness domains for functionals as in \eqref{dirienesub}. 
These issues clearly obstruct the use of classical techniques to study minimization properties of \eqref{dirienesub}, preventing for instance the $L^2$-lower semicontinuity of its Euclidean extension to $L^2(\Om),$ namely
\begin{align*}
    \displaystyle{\begin{cases}
    \int_\Omega \left(|X_1u|^2+|X_2 u|^2\right)\,{\rm d}x&\text{ if }u\in W^{1,2}(\Om),\\
    \infty&\text{ if }u\in L^2(\Omega)\setminus W^{1,2}(\Om).
\end{cases}}
\end{align*}

To overcome these obstacles, building on the foundational work of G.B. Folland and E.M. Stein \cite{MR0657581}, the correct functional framework has been developed by B. Franchi, R. Serapioni and F. Serra Cassano \cite{MR1437714,MR1448000} and by N. Garofalo and D.M. Nhieu \cite{MR1404326,GN}. 
Precisely, when $1\leq p<\infty$, the study of anisotropic functionals in the greater generality of \eqref{introlimitfun} can be carried out via the introduction of suitable anisotropic Sobolev and $BV$ spaces, say $W^{1,p}_X(\Om)$ and $BV_X(\Om)$, emerging as the natural domains of functionals as in \eqref{introlimitfun}. 
Although this construction is possible for arbitrary anisotropies, the specific structure of $X$ may lead to good approximation results \emph{à la} Meyers-Serrin, as well as to Poincaré inequalities and Rellich-Kondrachov compactness properties. We refer to \Cref{Section2} for more specific insights. When instead $p=\infty$, the reader is referred to \cite{PVW,MR2247884,MR2546006} for some anisotropic $L^\infty$-variational problems.

\medskip
Asymptotic results in the \emph{static} case as in \eqref{introlimitfun}, i.e. when the anisotropy $X$ is fixed, started in \cite{MR4108409,MR4504133,MR4054935} to generalize classical integral representation and $\Gamma$-convergence results due to G. Buttazzo and G. Dal Maso \cite{MR0839727}. These results were later extended in \cite{MR4609808,MR4566142} to complete integral functionals of the form
\[
\int_\Om f(x,u(x),Xu(x))\,{\rm d}x,\quad u\in W^{1,p}_{X}(\Om),
\]
generalizing the Euclidean counterparts \cite{MR0583636,MR0794824}.
In the above-mentioned contributions, the authors assume a full-rank condition on the family $X$, the so-called \emph{Linear Independence Condition (LIC)}.
This technical assumption has been definitely removed by the third-named author in \cite{withoutlic}. Thanks to these results, variational properties for static anisotropic functionals are well-understood and available in great generality.  

\medskip
In this paper we extend the results obtained in the aforementioned literature by considering \emph{moving anisotropies}. Namely, a static anisotropy is replaced by a sequence $(X^h)_h$ converging to a limit anisotropy $X$. 
Although our setting will be fairly more general, it models, as a particular case, the regularization of a sub-Riemannian structure via Riemannian approximants. As an instance, consider the 
anisotropy $X=(X_1,X_2)$ introduced in \eqref{anisotropiheisenberg} to model the first sub-Riemannian Heisenberg group. Its natural sub-Riemannian structure, say $g$, can be 
approximated by a sequence of Riemannian metrics, say $g_h$, each induced by the family of vector fields $X^h=(X_1,X_2,X_3^h)$, where 
\begin{equation*}
    X_3^h=\frac{1}{h}\frac{\partial}{\partial x_3}.
\end{equation*}
The benefits of such an approximation are manifold. For instance,  we stress that the natural approximating operators related to \eqref{laplaciansub}, namely
\begin{equation*}
      \left(X_1\right)^2u+\left(X_2\right)^2 u+\left(X_3^h\right)^2u,
\end{equation*}
are precisely the \emph{Laplace-Beltrami} operators associated with the Riemannian structures $g_h$. In turn, they are uniformly elliptic operators, their Dirichlet energies are coercive, and their associated functional frameworks boil down to the Euclidean ones.
We refer e.g. to \cite{MR2312336,MR2354992} for further insights, and to \cite{MR2774306,MR2262784,MR4761954,MR2043961,pozuelo2024existence} for a non-exhaustive list of applications of this approximation technique.

\medskip
The study of asymptotic properties of functionals of the form 
\begin{align}\label{funzionaligammaintro}
    F_h(u)=\displaystyle{\begin{cases}
    \int_\Om f_h(x,X^hu(x)){\rm d}x&\text{ if }u\in W^{1,p}_{X^h}(\Om),\\
    \infty&\text{ if }u\in L^p(\Omega)\setminus W^{1,p}_{X^h}(\Om),
\end{cases}}
\end{align}
when, say, $1<p<\infty$, requires some preliminary precautions. Indeed, contrarily to the static case, there are no \emph{a priori} relations between the finiteness domains of each functional in \eqref{funzionaligammaintro}.
Indeed, the latter may naturally vary according to the associated anisotropy. To this aim, it is important to understand how the approximating Sobolev spaces relate to the limit one, as discussed thoroughly in \Cref{Section3}. 

\medskip
Prescribing Dirichlet boundary conditions is also a delicate matter. Indeed, just as any finiteness domain in \eqref{funzionaligammaintro} depends on the associated anisotropy, it is natural in this setting to prescribe, for each fixed anisotropy $X^h$, a different boundary condition, say $\varphi_h\in W^{1,p}_{X_h}(\Om)$, and to consider sequences of functionals 
\begin{align}\label{funzionaligammaintrodiri}
    F^{\varphi_h}_h(u)=\displaystyle{\begin{cases}
    \int_\Om f_h(x,X^hu(x)){\rm d}x&\text{ if }u\in W^{1,p}_{X^h,\varphi_h}(\Om),\\
    \infty&\text{ if }u\in L^p(\Omega)\setminus W^{1,p}_{X^h,\varphi_h}(\Om),
\end{cases}}
\end{align}
where $W^{1,p}_{X^h,\varphi_h}(\Om)$ is the correct affine Sobolev space associated with the pair $(X^h,\varphi_h)$. We stress that, although moving boundary conditions may also be considered in the static case, in the moving one this choice is forced by the very structure of the problem. In this regard, we will discuss the right convergence assumptions on $(\varphi_h)_h$ to the limit boundary condition, say $\varphi$, in order to suitably relate the approximating and the limit affine Sobolev spaces (cf. \Cref{Section3}).

\medskip
A third important novelty in considering sequences as in \eqref{funzionaligammaintro} and \eqref{funzionaligammaintrodiri} consists in the fact that, unlike in the static case, possible coercivity and boundedness conditions typically depend on the anisotropy, and are thus not uniform along the sequence. Indeed, we will deal with growth conditions of the form 
\begin{equation}\label{movinggrowthintro}
     d_1  \int_\Om |X^hu(x)|^p{\rm d}x\leq \int_\Om f_h(x,X^hu(x)){\rm d}x\leq  \int_\Om a(x)\,{\rm d}x+d_2\int_\Om |X^hu(x)|^p{\rm d}x,
\end{equation}
where $a\in L^1(\Om)$ and $0<d_1\leq d_2$ are fixed. This lack of uniformity, as discussed thoroughly in \Cref{Section6}, prevents the use of some classical tools to establish $\Gamma$-compactness properties (cf. \Cref{spiegoneuno}), and requires a careful analysis of the asymptotic behaviour of the partial norms $\|X^hu\|_{L^p(\Om,\R^m)}$.

\medskip
A first fundamental step in handling these issues, as discussed in \Cref{Section4}, consists in establishing Meyers-Serrin type approximation results that keep track of the approximating sequence of anisotropies (cf. \Cref{meyser}). Roughly speaking, under uniform convergence of the anisotropies and mild regularity assumptions on $\Om$, functions in $W^{1,p}_X(\Om)$ can be approximated by a sequence of smooth functions $(u_h)_h$, where each $u_h\in W^{1,p}_{X^h}(\Om)$, in the sense that
\begin{equation*}
      \text{$u_h\to u$ strongly in $L^p(\Om)$}\qquad\text{and}\qquad\text{$X^hu_h\to Xu$ strongly in $L^p(\Om;\R^m)$}.
\end{equation*}
 To this aim,  we introduce a non-standard mollification technique (cf. \eqref{modulocontmatrici} and \eqref{mollificatorigiusti}), whose convergence rate is appropriately related to the convergence rate of the anisotropies. This approximation property, as explained in \Cref{Section6}, is crucial to deal with moving growth conditions as in \eqref{movinggrowthintro}. In the natural classes of anisotropies we will consider, we discuss similar approximation properties for affine Sobolev spaces  (cf. \Cref{stranachiusura} and \Cref{MeySerAffine}).
%

\medskip
In \cref{Section5} we establish Rellich-Kondrachov compactness properties and Poincaré inequalities that, again, keep into account the entire sequence $(X^h)_h$ (cf. \Cref{nuovothmcompattezza} and \Cref{convrayleight2}). 
Precisely, under natural assumptions on $X$ guaranteeing the validity of these properties in the limit, we show that such properties extend uniformly along $(X^h)_h$. We emphasize that such uniformity is fundamental in bypassing the non-uniformity of the coercivity assumptions as in \eqref{movinggrowthintro}. Indeed, a careful combination of the previous results will still provide, as in the static setting, boundedness of minimizing sequences and convergence of minima and minimizers (cf. \Cref{minimaminimizerthm}).

\medskip
The core of this work is \cref{Section6}, where we provide several results related to $\Gamma$-convergence.
Owing to \Cref{Section4}, we first prove $\Gamma$-convergence of the partial norms $\|X^hu\|_{L^p(\Om,\R^m)}$ to the limit one $\|Xu\|_{L^p(\Om,\R^m)}$ (cf. \Cref{normgammaconv}). 
As already mentioned, this first feature is crucial in the establishment of a general $\Gamma$-compactness theorem for functionals as in \eqref{funzionaligammaintro} (cf. \cref{gammathmwithproof}). 
Next, we keep into account Dirichlet boundary data, obtaining $\Gamma$-compactness for sequences of functionals as in \eqref{funzionaligammaintrodiri} (cf. \Cref{mainthmdirichletbc}). 
As pointed out in \Cref{spiegonedue}, the fact that the boundary data necessarily evolves with the anisotropies causes some challenges in the identification of the appropriate recovery sequences, as the latter need to keep into account the above-mentioned evolution. In this regard, a key tool is provided by the aforementioned uniform approximation result for affine Sobolev spaces, \Cref{MeySerAffine}.
As already highlighted, the results of \Cref{Section5} allow us to prove the convergence of minima and minimizers for $\Gamma$-converging sequences of functionals, as in \eqref{funzionaligammaintrodiri}, as stated in \Cref{minimaminimizerthm}.

\medskip
We conclude \Cref{Section6} by specializing our main $\Gamma$-compactness result, \Cref{gammathmwithproof}, to the class of anisotropic \emph{quadratic forms} 
\begin{equation*}
    F_h(u)=\begin{cases}
    \int_\Om\langle A^h(x)X^hu(x),X^hu(x)\rangle_{\R^m}{\rm d}x\quad&\text{if }u\in W^{1,2}_{X^h}(\Om),\\
    \infty\quad&\text{if }u\in L^2(\Om)\setminus W^{1,2}_{X^h}(\Om),
    \end{cases}
\end{equation*}
whence establishing $\Gamma$-compactness in this subclass (cf. \Cref{compactnessL2perHconvergenza}). The study of the latter is particularly relevant in connection with anisotropic sequences of symmetric, linear differential problems 
%
%
%
\begin{equation}\label{euqationintrogeneral}
    \begin{cases}
\mu u_h-{\rm div}_{X^h}(A^h(x)X^hu_h)=f&\text{in }\Omega,\\
u_h\in W^{1,2}_{X^h,\varphi_h}(\Omega),
\end{cases}
\end{equation}
as well as with their asymptotic behaviour. 
We stress that the latter are clear generalisations of the prototypical differential problems associated with \eqref{laplaciansub}. In \Cref{Section7}, we adapt the classical notion of $H$-convergence for keeping into account the moving anisotropic framework (cf. \Cref{hconvdef}), and we establish a related $H$-compactness property (cf. \Cref{teoremahcompattezza455y65hy}). This will follow as a corollary of the previous results, as well as of some additional convergence properties established in \Cref{thmmomentaconv} to deal with the so-called \emph{momenta} associated with \eqref{euqationintrogeneral}.
%
%
\medskip

The paper is structured as follows.  
In \cref{Section2}, we recall definitions and properties of the static anisotropic functional setting.  
In \cref{Section3}, we introduce the relevant classes of moving anisotropies we are interested in, and we discuss some preliminary functional results, including the characterization of affine Sobolev spaces (\Cref{MeySerAffine}).  
Later, in \cref{Section4}, we establish uniform Meyers-Serrin type approximations for anisotropic Sobolev (\cref{meyser}) and BV (\cref{thmanzellottigiaqervgr}) spaces.
In \cref{Section5}, we state and prove a uniform Rellich-Kondrachov compactness result (\cref{nuovothmcompattezza}), as well as a uniform Poincaré inequality (\cref{convrayleight2}).  
In \cref{Section6}, we first show the main $\Gamma$-compactness result, \cref{gammathmwithproof}, as well as its extension to functionals including Dirichlet boundary conditions (\cref{mainthmdirichletbc}) and quadratic forms (\cref{compactnessL2perHconvergenza}).  
Moreover, we provide convergence of minima, minimizers (\cref{minimaminimizerthm}) and momenta (\cref{thmmomentaconv}).  
As a corollary, in \Cref{Section7} we obtain the $H$-compactness result, \cref{teoremahcompattezza455y65hy}.  


\section{Anisotropic functional frameworks}\label{Section2}
\subsection*{Main notations}
If no ambiguity arises, we let $\infty=+\infty$ and $\N=~\N~\setminus~\{0\}$. If $1\leq p\leq\infty$, we denote by $p'$ the H\"older-conjugate exponent of $p$. We write $(a_h)_h\subseteq (A_h)_h$ meaning that $a_h\in A_h$ for any $h\in\N$. We fix $m,n\in\N$ such that $0<m\leq n$. If $A,B\subseteq\R^n$, we let $A+B=\{a+b\,:\,a\in A,\,b\in B\}$.
For $\alpha,\beta\in\N$, we denote by $M(\alpha,\beta)$ the set of matrices with $\alpha$ rows and $\beta$ columns.
Vectors in $\R^\alpha$ are seen as matrices in $M(\alpha,1)$. We denote by either $\cdot $ or $\langle\cdot,\cdot\rangle_{\R^\alpha}$ the scalar product between vectors in $\R^\alpha$.
Fixed an open and bounded set $\Om\subseteq\R^n$, we denote by $\mA$ the class of all the open subsets of $\Om$, and for a function $u$ defined in $\Omega$, we may tacitly assume that it is defined on $\R^n$ by extending it to be zero outside $\Om$. Finally, we denote by $Du$ the (distributional) Euclidean gradient of $u$. 
\subsection{Anisotropies} \label{section21}


Throughout the paper, $\Om$ is an open and bounded subset of $\R^n$. Given a family $\X\coloneqq(X_1,\ldots,X_m)$ of Lipschitz continuous vector fields on $\Omega$, such that
\[
X_j=\sum_{i=1}^nc_{j,i}\frac{\partial}{\partial x_i},\quad\text{with }c_{j,i}\in\Lip(\Om)
\]
for any $j=1,\ldots,m$ and any $i=1,\ldots,n$, we denote by $\C(x)\coloneqq[c_{j,i}(x)]_{\substack{{i=1,\dots,n}\\{j=1,\dots,m}}}$ the $m\times n$ \emph{coefficient matrix} associated with $X$. We may refer to $X$ as \emph{(static) anisotropy}, or \emph{$X$-gradient}. This notation is motivated by the following definition.
\begin{definition}\label{xgraddef}
Let $u\in L^1_{\rm loc}(\Om)$, let $V\in L^1_{\rm loc}(\Om;\rr^m)$.
\begin{enumerate}
    \item The $X$-gradient of $u$ is the $m$-dimensional distribution defined by 
    \begin{equation*}
        Xu(\varphi)= -\int_\Om\,u\,\divv(\C^T\varphi)\,{\mathrm d}x\qquad\text{for any $\varphi\in \mathbf{C}^\infty _c(\Om;\R^m)$.}
    \end{equation*}
    \item The $X$-divergence of $V$ is the $1$-dimensional distribution defined by 
    \begin{equation*}
        \divv_X(V)(\varphi)= -\int_\Om \left\langle V,\C D\varphi\right\rangle_{\rr^m}\,{\mathrm d}x\qquad\text{for any $\varphi\in \mathbf{C}^\infty _c(\Om)$.}
    \end{equation*}
\end{enumerate}
\end{definition}
Clearly, in the isotropic case, the above definition returns the usual notions of gradient and divergence.
In what follows, we will sometimes consider the following classes of $X$-gradients.
\begin{definition}\label{deficlassminipoinc}
We let $\mathcal{M}=\mathcal{M}(\Om_0,C_d,c,C)$ be the class of $X$-gradients that are defined and Lipschitz continuous on an open set $\Omega_0\supseteq\overline{\Omega}$, and that satisfy the following three conditions.
\begin{itemize}
    \item [$(\mathrm{H}.1)$] Let $d$ be the \emph{Carnot-Carath\'eodory distance} induced by $X$ on $\Om_0$ (cf. e.g. \cite{MR1421823}). 
Then $d$ is finite and continuous with respect to the usual topology of $\R^n$.
 \item [$(\mathrm{H}.2)$] 
There exist positive constants $R, C_d$ such that 
\begin{equation*}\label{h2ip}
    |B_d(x,2r)| \leqslant C_d \, |B_d(x,r)|
\end{equation*}
for any $x \in \overline{\Omega}$ and $r \leqslant R$. 
Here, $B_d(x,r)$ denotes the open metric ball with respect to $d$.
 \item [$(\mathrm{H}.3)$] 
After decreasing $R$ in $(\mathrm{H}.2)$ if necessary, there exist positive constants $c, C$ such that for every open ball $B\coloneqq B_d(\bar{x},r)$, with $\bar{x}\in\Omega$, and $0\les r\les R$  such that $cB\coloneqq B_d(\bar{x},cr)\subseteq\Omega_0$, and for every $u\in \mathrm{Lip}(\overline{c B})$ and $x \in \overline{B}$,
\begin{equation*}\label{h3ip}
    \left|u(x)-\frac{1}{|B|}\int_B u(y)\,{\rm d}y\right|\leq C \int_{cB} |Xu(y)| \frac{d(x,y)}{|B_d(x,d(x,y))|}\,{\rm d}y.
\end{equation*}
\end{itemize}
\end{definition}
Several important families belong to $\mathcal{M}(\Om_0,C_d,c,C)$.
In particular, \emph{H\"ormander vector fields}, i.e. smooth vector fields for which the rank of the Lie algebra generated by $X_1,\ldots, X_m$ equals $n$ at any point of $\Omega_0$, belong to $\mathcal{M}(\Om_0,C_d,c,C)$.
Among H\"ormander vector fields, we recall the following instances.
\begin{itemize}
    \item[(i)] \emph{Euclidean vector fields:} $X=\left(\frac{\partial}{\partial x_1},\dots,\frac{\partial}{\partial x_n}\right)$, defined on $\R^n$.
    \item[(ii)] \emph{Grushin vector fields:} $X=\left(\frac{\partial}{\partial x_1},x_1\frac{\partial}{\partial x_2}\right)$, defined on $\R^2$.
    \item[(iii)] \emph{Heisenberg vector fields:} $X=\left(\frac{\partial}{\partial x_1}+x_2\frac{\partial}{\partial x_3},\frac{\partial}{\partial x_2}-x_1\frac{\partial}{\partial x_3}\right)$, defined on $\R^3$.
\end{itemize}
In particular, Heisenberg vector fields are prototypical \emph{Carnot groups vector fields} (cf. e.g. \cite{MR2363343}).
It is worth noting that all Carnot groups vector fields are, by definition, H\"ormander vector fields, and hence belong to the class $\mathcal{M}(\Om_0,C_d,c,C)$.
For further insights on this topic and for examples of families of vector fields in the class $\mathcal{M}(\Om_0,C_d,c,C)$ that do not satisfy the H\"ormander condition, the reader is referred to \cite{MR4504133}.


\subsection{Sobolev spaces} 
Owing to \Cref{xgraddef}, we introduce the main functional framework.

\begin{definition}
The anisotropic Sobolev spaces associated to an $X$-gradient are defined by
\begin{align*}
    \anso{\Om}&\coloneqq\{u\in L^p(\Om)\,:\,Xu\in L^p(\Om;\R^m)\},\\
    H^{1,p}_X(\Om)&\coloneqq\overline{\mathbf{C}^\infty(\Om)\cap W^{1,p}_X(\Om)}^{\|\cdot\|_{W^{1,p}_X(\Om)}},\\
    W^{1,p}_{X,0}(\Om)&\coloneqq\overline{\mathbf{C}^\infty_c(\Om)\cap W^{1,p}_X(\Om)}^{\|\cdot\|_{W^{1,p}_X(\Om)}},\\
    W^{1,p}_{X,\varphi}(\Om)&\coloneqq\{u\in W^{1,p}_X(\Om)\,:\,u-\varphi\in W^{1,p}_{X,0}(\Om)\}\qquad\text{ for any }\varphi\in W^{1,p}_X(\Om).
\end{align*}
\end{definition}
It is well-known (cf. \cite{MR0657581}) that the vector space $\anso{\Om}$, endowed with the norm
\[
\Vert u\Vert_{\anso{\Om}}\coloneqq\Vert u\Vert_{L^p(\Om)}+\Vert Xu\Vert_{L^p(\Om;\R^m)},
\]
is a Banach space for any $1\leq p<+\infty$, reflexive when $1<p<+\infty$. 
The Lipschitz continuity assumption on the family $\X$ ensures that the Euclidean Sobolev space
$W^{1,p}(\Om)$ embeds continuously into $W^{1,p}_X(\Om)$ (cf. \cite{MR4609808,MR4108409}).
Precisely, for any $u\in W^{1,p}(\Omega)$, the $X$-gradient admits the Euclidean representation
\begin{equation}\label{euclrapprgrad}
    Xu(x)=\C(x)Du(x)\qquad\text{for a.e. $x\in\Om$.}
\end{equation}
We also recall that, as for classical Sobolev spaces, $H^{1,p}_X(\Om)\subseteq\anso{\Om}$.
The classical result of Meyers and Serrin 
is still valid in this anisotropic setting and it was proved, independently, in \cite{MR1437714} and \cite{MR1404326}.
We state the result in what follows, and refer again the interested reader to \cite{MR4504133} for further discussions on this topic.
\begin{theorem}\label{MeySer}
For any $1\leq p <\infty$, it holds that $H^{1,p}_X(\Om)=\anso{\Om}.$
\end{theorem}
If, in addition, $X$ belongs to the class $\mathcal{M}(\Omega_0,C_d,c,C)$, the following Poincar\'e inequality and Rellich-Kondrachov theorem hold.
We refer to \cite[Proposition 2.16]{MR4504133} and \cite[Theorem 3.4]{MR1448000}, respectively.
\begin{proposition}[Poincar\'e inequality]\label{poincaprelprop}
Let $1\leq p<\infty$ and let $X\in\mathcal{M}(\Omega_0,C_d,c,C)$.
Then, there exists a positive constant $c_{p,\Omega}=c_{p,\Om}(p,\Om)$ such that 
\begin{equation*}
    c_{p,\Omega}\,\int_\Om|u|^p\,{\rm d}x\le\int_\Om |Xu|^p\,{\rm d}x\quad\text{for any $u\in W^{1,p}_{X,0}(\Omega)$.}
\end{equation*}
\end{proposition}
\begin{theorem}[Rellich-Kondrachov]\label{relkon}
For $1\leq p<\infty$, $W^{1,p}_{X,0}(\Omega)$ compactly embeds in $L^p(\Omega)$.
\end{theorem}

\section{Moving anisotropies}\label{Section3}


We fix a sequence of $X$-gradients $(X^h)_h$ with $m$ components $X^h = (X_1^h, \ldots, X_m^h)$, identified by a sequence of matrices $(\C^h)_h$ as in \Cref{section21}.
Throughout the paper, we assume that there exists a limit family of vector fields $X=(X_1,\ldots,X_m)$, identified by a Lipschitz continuous coefficient matrix $\C$, such that 
\begin{equation*}\label{a0}\tag{A}
     \C^h\to \C\text{ uniformly on }\Om.
\end{equation*}
We may sometimes use the index $\infty$ to refer to $X$. As already mentioned, \eqref{a0} will be sufficient to infer good approximation properties of anisotropic Sobolev functions. Instead, most of the $\Gamma$-compactness properties will be achieved for the following two classes of $X$-gradients.
\begin{definition}
    We say that $(X^h)_h\in\classuno$ if \eqref{a0} holds and if there exists a sequence $(K_h)_h\subseteq\rr$ such that $K_h\to 1$ as $h\to\infty$ and 
    \begin{equation*}\label{partshape21}\tag{S.1}
        |\C(x)\xi|\leq K_h|\C^h(x)\xi|\qquad\text{for any $x\in\Om$, $\xi\in\R^n$ and $h\in\N$.}
    \end{equation*}
\end{definition}
\begin{definition}\label{classduedef}
We say that $(X^h)_h\in\classdue$ if \eqref{a0} holds and if 
\begin{equation*}\label{ublip}\tag{S.2}
    (\C^h)_h\text{ is uniformly bounded in }\Lip(\Om,M(m,n)).
\end{equation*}
\end{definition}

 \begin{remark}
     As pointed out in the introduction, $\classuno$ models those frameworks where a sub-Riemannian structure induced by $X$ is approximated by a sequence of converging Riemannian structures. Indeed, a relevant example of vector fields satisfying \eqref{partshape21} is given by
     \begin{equation}\label{partshape21vecchio}
    \C^h=\left[ \begin{array}{c}
	\C  \\
	\D^h
\end{array}
\right]\qquad\text{and}\qquad \C=\left[ \begin{array}{c}
	\C  \\
	0
\end{array}
\right],
\end{equation}
with a natural abuse of notation in the second property. In the typical sub-Riemannian framework, $\C$ encodes the sub-Riemannian environment, and $\C^h$ models its approximating Riemannian structure. We point out that \eqref{partshape21} is less restrictive than \eqref{partshape21vecchio}. For instance, the vector fields
$$X^h=\left(1+\frac 1h\right)\frac{\partial}{\partial x}\qquad \text{on }(0,1)\subseteq\rr$$ belongs \eqref{partshape21} but do not have the structure \eqref{partshape21vecchio}.
     Instead, $\classdue$ allows anisotropies with arbitrary shapes, provided the uniform bound \eqref{ublip} is satisfied. 
A relevant example, exploited e.g. in \cite{MR2663002} in connection with applications to image processing, is provided by the sequence of vector fields
\[
 X_1^h=\cos\left(\frac{x_2}{h}\right)\frac{\partial}{\partial x_1}
      +h\sin\left(\frac{x_2}{h}\right)\frac{\partial}{\partial x_3},
 \qquad X_2^h=\frac{\partial}{\partial x_2}\qquad \text{ on }\rr^3.
\]
On every bounded open set, they converge uniformly to 
\[
 X_1=\frac{\partial}{\partial x_1}
      +x_2\frac{\partial}{\partial x_3},
 \qquad X_2=\frac{\partial}{\partial x_2}.
\]
It is easy to verify that $(X^h)_h\in\classdue\setminus \classuno$. Moreover, $X$ is a smooth family satisfying the H\"ormander condition (i.e. the forthcoming property \eqref{hormanderlimit}).
 \end{remark}
\begin{remark}\label{banachalaogluremark}
   If $(X^h)_h\in\classdue$, 
  by \eqref{ublip}, we may always assume that \begin{equation*}
    \frac{\partial c^h_{j,i}}{\partial x_k}\to  \frac{\partial c_{j,i}}{\partial x_k}\quad\text{weakly$^\star$ in $L^\infty(\Om)$}
\end{equation*}
for any $j=1,\ldots,m$ and any $i,k=1,\ldots,n$.
\end{remark}
\begin{remark}
We stress that $\classuno\cap\classdue\neq \emptyset$. Nevertheless, neither $\classuno\subseteq\classdue$ nor $\classdue\subseteq\classuno$. We already know that $\classdue\not\subseteq\classuno$. To see that $\classuno\not\subseteq\classdue$, consider the following instance. Let $\Om=(-1,1)^2\subseteq\R^2_{(x_1,x_2)}$. For any $h\in\N$, let $X^h=(X_1,X^h_2)$, where 
\begin{equation*}
    X_1=\frac{\partial}{\partial x_1},\qquad X^h_2=f_h(x_2)\frac{\partial}{\partial x_2} \qquad\text{and}\qquad    f_h(x_2) = \begin{cases}
    -\frac{1}{h} & \text{if } x_2 \in (-1, -\frac{1}{h^2}), \\
    h x_2 & \text{if } x_2 \in [-\frac{1}{h^2}, \frac{1}{h^2}], \\
    \frac{1}{h} & \text{if } x_2 \in (\frac{1}{h^2},1).
    \end{cases}
\end{equation*}
A simple computation reveals that $(X^h)_h\in\classuno\setminus\classdue$.
\end{remark}
In order to address compactness properties, we will need to further refine our assumptions. We consider the following two conditions on the limit anisotropy.
\begin{equation*}\label{rk}\tag{B.1}
      X\in\mathcal{M}=\mathcal{M}(\Om_0,C_d,c,C),
\end{equation*}
\begin{equation*}\label{hormanderlimit}\tag{B.2}
      X\text{ is smooth and satisfies the H\"ormander condition on an open set $\Om_0$ such that  }\overline\Om\subseteq\Om_0.
\end{equation*}
\begin{definition}
 We say that:
\begin{enumerate}
    \item [(i)] $(X^h)_h\in\classunoraff$ if $(X^h)_h\in\classuno$ and \eqref{rk} holds;
      \item [(ii)] $(X^h)_h\in\classdueraff$ if $(X^h)_h\in\classdue$ and \eqref{hormanderlimit} holds.
\end{enumerate}
\end{definition}
In the rest of this section, we investigate some preliminary properties of moving anisotropies. 
We introduce the \emph{anisotropic Rayleigh quotients}. If $(X^h)_h$ satisfies \eqref{a0}, we let 
\begin{equation*}
    \mathcal R_h\coloneqq\inf\left\{\frac
    {\int_\Om|X^hu|^p\,{\rm d}x}{\int_\Om|u|^p\,{\rm d}x}\,:\,u\in W^{1,p}_{X_h,0}(\Om),\,u\neq 0\right\}  =\inf\left\{\frac
    {\int_\Om|X^hu|^p\,{\rm d}x}{\int_\Om|u|^p\,{\rm d}x}\,:\,u\in \mathbf{C}^\infty _c(\Om),\,u\neq 0\right\}  
\end{equation*}
for any $h\in\N$, and we let $\mathcal R$ be the Rayleigh quotient associated with $X$.
Above, the second equality follows by definition of $W^{1,p}_{X^h,0}(\Om)$.
\subsection{Preliminary functional properties}
The following two propositions describe the relations between the approximating Sobolev spaces and the limit one in $\classuno$ and $\classdue$, respectively.
%
\begin{proposition}\label[proposition]{primacasofacile}
Let $(X^h)_h\in\classuno$. Let $1\leq p<\infty$. Then, for any $h\in\N$,
\[
 W^{1,p}_{X^h}(\Om)\subseteq W^{1,p}_X(\Om)\qquad\text{and}\qquad
 W^{1,p}_{X^h,0}(\Om)\subseteq W^{1,p}_{X,0}(\Om).
\]
Moreover,
\begin{equation}\label{faciliprop}
 |Xu|\leq K_h|X^hu|\quad\text{a.e. in $\Om$, for any $u\in W^{1,p}_{X^h}(\Om)$.}
\end{equation}
If $1<p<\infty$, and if $(u_h)_h\subseteq W^{1,p}_{X^h}(\Om)$ satisfies $u_h\to u$ weakly in $L^p(\Om)$ and $\displaystyle{\liminf_{h\to\infty}}\|X^hu_h\|_{L^p(\Om;\R^m)}<\infty$, then $u\in W^{1,p}_X(\Om)$ and
\begin{equation*}
    \|Xu\|_{L^p(\Om;\R^m)}\leq\liminf_{h\to\infty}\|X^hu_h\|_{L^p(\Om;\R^m)}.
\end{equation*}
\end{proposition}
\begin{proof}

Fix $h\in\N$ and let $u\in W^{1,p}_{X^h}(\Om)$. By \Cref{MeySer},
there exists a sequence $(u_k)_k$ of smooth functions converging to $u$
in $W^{1,p}_{X^h}(\Om)$. By \eqref{partshape21}, $|Xu_k|\leq K_h|X^hu_k|$ and $|X(u_k-u_j)|\leq K_h|X^h(u_k-u_j)|$ for every $j,k$.
Therefore, $(Xu_k)_k$
converges in $L^p(\Om;\R^m)$. Let $\varphi\in\mathbf{C}^\infty_c(\Om;\R^m)$.
Then
\begin{equation*}
    -\int_\Om u\divv(\C^T\varphi)\,{\rm d}x
    =-\lim_{k\to\infty}\int_\Om u_k\divv(\C^T\varphi)\,{\rm d}x
    =\lim_{k\to\infty}\int_\Om Xu_k\cdot\varphi\,{\rm d}x,
\end{equation*}
whence $u\in W^{1,p}_X(\Om)$ and $Xu_k\to Xu$ strongly in $L^p(\Om;\R^m)$.
Then \eqref{faciliprop} is straightforward. Moreover, if $u\in W^{1,p}_{X^h,0}(\Om)$, the functions
$u_k$ can be chosen in $\mathbf{C}^\infty_c(\Om)$, whence
$u\in W^{1,p}_{X,0}(\Om)$.
Let $1<p<\infty$.
Let $(u_h)_h$ and $u$ be as in the statement.
By reflexivity and \eqref{faciliprop}, there exists $v\in L^p(\Om;\R^m)$
such that, up to a subsequence, $Xu_h\to v$ weakly
in $L^p(\Om;\R^m)$. Then $u\in W^{1,p}_X(\Om)$ and $v=Xu$. The thesis
follows by the weak lower semicontinuity of the $L^p$-norm.
\end{proof}
\begin{proposition}\label{limitinsobolev}
Let $(X^h)_h\in\classdue$. Let $1<p<\infty$.  
Let $(u_h)_h\subseteq W^{1,p}_{X^h}(\Om)$.
\begin{enumerate}
    \item If $u_h\to u$ weakly in $L^p(\Om)$ and $X^h u_h\to v$ weakly in $L^p(\Om;\R^m)$, then $ u\in W^{1,p}_X(\Om)$.
    \item If in addition $u_h\to u$ strongly in $L^p(\Om)$, then $v=Xu$ and 
  \begin{equation}\label{liminfinequalityfvrgvgtrbvgrbfbvgfrbrb}
    \|Xu\|_{L^p(\Om;\R^m)}\leq\liminf_{h\to\infty}\|X^hu_h\|_{L^p(\Om;\R^m)}.
\end{equation}
\end{enumerate}

\end{proposition}
\begin{proof}
Let $p'$ the H\"older-conjugate exponent of $p$. Let $(u_h)_h\subseteq W^{1,p}_{X^h}(\Om)$ be such that $u_h\to u$ weakly in $L^p(\Om)$.
By 
\eqref{ublip}, 
\begin{equation}\label{weaklyhps2}
    u_hc^h_{j,i}\to u c_{j,i}\text{ weakly in }L^p(\Om)
\end{equation}
for any $i=1,\ldots,n$ and any $j=1,\dots, m$. Fix $\varphi\in \mathbf{C}^\infty _c(\Om;\R^m)$. Then
\begin{equation*}
\begin{split}
   -\int_\Om u &\diver (\C^T\varphi)\,{\rm d} x =-\int_\Om u\sum_{i=1}^n\sum_{j=1}^mc_{j,i}\frac{\partial \varphi_j}{\partial x_i}\,{\rm d}x-\int_\Om u\sum_{i=1}^n\sum_{j=1}^m\frac{\partial c_{j,i}}{\partial x_i}\varphi_j\,{\rm d}x\\
    &\overset{\eqref{weaklyhps2}}{=}-\lim_{h\to\infty}\int_\Om u_h\sum_{i=1}^n\sum_{j=1}^mc^h_{j,i}\frac{\partial \varphi_j}{\partial x_i}\,{\rm d}x-\int_\Om u\sum_{i=1}^n\sum_{j=1}^m\frac{\partial c_{j,i}}{\partial x_i}\varphi_j\,{\rm d}x\\    &=\lim_{h\to\infty}\left(-\int_\Om u_h\divv((\C^h)^T\varphi)\,{\rm d}x+\int_\Om u_h\sum_{i=1}^n\sum_{j=1}^m\frac{\partial c^h_{j,i}}{\partial x_i}\varphi_j\,{\rm d}x\right)-\int_\Om u\sum_{i=1}^n\sum_{j=1}^m\frac{\partial c_{j,i}}{\partial x_i}\varphi_j\,{\rm d}x\\
    &= \lim_{h\to\infty}\left(\int_\Om \langle X^hu_h,\varphi\rangle_{\R^m}\,{\rm d}x+\int_\Om u_h\sum_{i=1}^n\sum_{j=1}^m\frac{\partial c^h_{j,i}}{\partial x_i}\varphi_j\,{\rm d}x\right)-\int_\Om u\sum_{i=1}^n\sum_{j=1}^m\frac{\partial c_{j,i}}{\partial x_i}\varphi_j\,{\rm d}x\\
    &= \int_\Om \langle v,\varphi\rangle_{\R^m}\,{\rm d}x+\lim_{h\to\infty}\int_\Om u_h\sum_{i=1}^n\sum_{j=1}^m\frac{\partial c^h_{j,i}}{\partial x_i}\varphi_j\,{\rm d}x-\int_\Om u\sum_{i=1}^n\sum_{j=1}^m\frac{\partial c_{j,i}}{\partial x_i}\varphi_j\,{\rm d}x.
\end{split}
\end{equation*}
By \eqref{ublip} and H\"older's inequality, if $\|\varphi\|_{L^{p'}(\Om;\R^m)}\leq 1$ the last line 
is bounded above independently of $\varphi.$ In particular, $u\in W^{1,p}_X(\Om)$. If in addition $u_h\to u$ strongly in $L^p(\Om)$, then the last two terms in the last line 
simplify by \eqref{ublip} and \Cref{banachalaogluremark}. Hence $v=Xu$ by definition. The thesis follows by the lower semicontinuity of the $L^p$-norm with respect to weak convergence.
\end{proof}
%
The following example shows that \eqref{liminfinequalityfvrgvgtrbvgrbfbvgfrbrb} is generally false under the sole condition \eqref{a0}.
\begin{example}\label[example]{esempiononvaleleimonfifr}
Fix $K>1$, let $\Om=(0,1)^2$ and  $\beta=\sqrt{K^2-1}$. For any $h\in\N_+$, define
\begin{equation*}
    \sigma_h(t)=\frac{1}{h}\tanh\!\big(h\sin(2\pi h^2t)\big),\qquad t\in\R,
\end{equation*}
and set 
\begin{equation*}
    \C^h =
    \begin{bmatrix}
        1 & \beta \sigma_h(x_2) \\
        0 & \sigma_h(x_2)
    \end{bmatrix},
    \qquad
    \C =
    \begin{bmatrix}
        1 & 0 \\
        0 & 0
    \end{bmatrix}.
\end{equation*}
Denote by $(X^h)_h$ and $X$ the corresponding families of anisotropies. Clearly, \eqref{a0} holds. Nevertheless, although 
\begin{equation*}
    |\C(x)\xi|\leq K|\C^h(x)\xi|\qquad \text{for any $x\in\Om$, $\xi\in\R^n$ and $h\in\N$,}
\end{equation*}
it is easy to show that $(X^h)_h\notin\classuno\cup\classdue$. Set
\[
    u(x)=x_1,\qquad u_h(x)=x_1-\frac{\beta h^2}{K^2}
                 \int_0^{x_2}\sigma_h(t)\,{\rm d}t .
\]
It is easy to check that, for any $t$, $\sigma_h(t+1/(2h^2))=-\sigma_h(t)$. Then $\sigma_h$
has zero average over each period of length $1/h^2$. For $x_2\in(0,1)$, let
$m=\lfloor h^2x_2\rfloor$. Then $0\leq x_2-m/h^2<1/h^2$ and, since
$|\sigma_h|\leq 1/h$, we have
\[
    \left|\int_0^{x_2}\sigma_h(t)\,{\rm d}t\right|
    =\left|\int_{m/h^2}^{x_2}\sigma_h(t)\,{\rm d}t\right|
    \leq \frac{1}{h}\left(x_2-\frac{m}{h^2}\right)
    \leq \frac{1}{h^3}.
\]
Multiplying by $\beta h^2/K^2$ gives
$\|u_h-u\|_\infty\leq\beta/(K^2h)$. Then, $u_h\to u$ strongly in $L^p(\Om)$ for every $1\leq p\leq\infty$.
 On the other hand, since $h^2\sigma_h^2\to1$ in $L^p(0,1)$ for every $1\leq p\les \infty$, 
\[
    X^hu_h
    =(1,0)-\frac{\beta}{K^2}\,h^2\sigma_h(x_2)^2\,(\beta,1)
    \longrightarrow \frac{1}{K^2}(1,-\beta)
    \qquad\text{in }L^p(\Om;\R^2)\qquad \text{for every $1\leq p\les\infty$}.
\]
Since $Xu=(1,0)$, we conclude that 
\[
    \lim_{h\to\infty}\|X^hu_h\|_{L^p(\Om;\R^2)}
    =\frac1K<1=\|Xu\|_{L^p(\Om;\R^2)}.
    \]
\end{example}
Our next result states the convergence of Rayleigh quotients when $(X^h)_h\in\classuno$.
\begin{proposition}\label[proposition]{convrayleighs1}
 Let $(X^h)_h\in\classuno$. Let $1<p<\infty$. Then
\begin{equation}\label{summa}
   \mathcal R\leq\mathcal K_h^p\mathcal R_h\text{ for any }h\in\N\qquad\text{and}\qquad\lim_{h\to\infty}\mathcal R_h=\mathcal R.
\end{equation}
\end{proposition}
\begin{proof}
Fix $u\in \mathbf{C}^\infty _c(\Om)$ such that $u\neq 0$. Then \Cref{primacasofacile} implies that
    \begin{equation*}
        \mathcal R\leq \frac{\int_\Om|Xu|^p\,{\rm d}x}{\int_\Om|u|^p\,{\rm d}x}\leq K_h^p\frac{\int_\Om|X^hu|^p\,{\rm d}x}{\int_\Om|u|^p\,{\rm d}x},
    \end{equation*}
whence the first property of \eqref{summa} follows.
 On the other hand, fix $u\in \mathbf{C}^\infty _c(\Om)$ such that $u\neq 0$. By \eqref{a0},
  $X^hu\to Xu$ uniformly on $\Om$, whence
    \begin{equation*}
        \limsup_{h\to\infty}\mathcal R_h\leq\limsup_{h\to\infty}\frac{\int_\Om|X^hu|^p\,{\rm d}x}{\int_\Om|u|^p\,{\rm d}x}=\frac{\int_\Om|Xu|^p\,{\rm d}x}{\int_\Om|u|^p\,{\rm d}x}.
    \end{equation*}
    Being $u\neq 0$ arbitrarily chosen in $\mathbf{C}^\infty _c(\Om)$, \eqref{summa} follows.
\end{proof}
Finally, we describe the relation between the approximating affine Sobolev spaces and the limit one.
\begin{proposition}\label[proposition]{stranachiusura}
Assume that either $(X^h)_h\in\classuno$ or $(X^h)_h\in\classdue$ and $\Om$ has Lipschitz boundary.
    Let $1 <p<\infty$. Let $(\varphi_h)_h\subseteq W^{1,p}_{X^h}(\Om)$, and let $\varphi\in W^{1,p}_{X}(\Om)$. Assume that
    \begin{equation}\label{quasichiusuradebolehp}
        \varphi_h\to\varphi\text{ weakly in }L^p(\Om)\qquad\text{and}\qquad X^h\varphi_h\to X\varphi\text{ weakly in }L^p(\Om;\R^m).
    \end{equation}
    Let $(u_h)_h\subseteq W^{1,p}_{X^h,\varphi_h}(\Om)$, and let $u\in W^{1,p}_{X}(\Om)$. 
    Assume that
    \begin{equation}\label{quasichiusuradebolehpsuu}
        u_h\to u\text{ weakly in }L^p(\Om)\qquad\text{and}\qquad X^hu_h\to Xu\text{ weakly in }L^p(\Om;\R^m).
    \end{equation}
    Then $u\in W^{1,p}_{X,\varphi}(\Om)$.
\end{proposition}
\begin{proof}
\textbf{Step 1.} By \eqref{quasichiusuradebolehp} and \eqref{quasichiusuradebolehpsuu},
replacing $u_h$ by $u_h-\varphi_h$ and $u$ by $u-\varphi$ reduces the
proof to the case $\varphi_h=\varphi=0$. Thus $u_h\in W^{1,p}_{X^h,0}(\Om)$ for any $h$. We show that $u\in W^{1,p}_{X,0}(\Om)$.\\
\noindent \textbf{Step 2.} Assume first that $(X^h)_h\in\classuno$. By \Cref{primacasofacile},
$u_h\in W^{1,p}_{X,0}(\Om)$ and $|Xu_h|\leq K_h|X^hu_h|$ a.e. in $\Om$.
By reflexivity, and since $u_h\to u$ weakly in $L^p(\Om)$, then $Xu_h\to Xu$ weakly in $L^p(\Om,\rr^m)$. Since $W^{1,p}_{X,0}(\Om)$ is sequentially weakly closed in $W^{1,p}_X(\Om)$, we conclude that $u\in W^{1,p}_{X,0}(\Om)$.\\
\noindent \textbf{Step 3.} Assume now that $(X^h)_h\in\classdue$ and that $\Om$ has Lipschitz
boundary. Extend $\C$ to a Lipschitz continuous matrix-valued function defined on $\R^n$. If $v$ is any function over $\Om$, denote by $\widetilde v$ its zero extension outside $\Om$. By \Cref{limitinsobolev}, $u\in W^{1,p}_X(\Om)$. We first show that, as distributions, 
$X\widetilde u=\widetilde{Xu}$ in $\R^n$. Fix $\psi\in\mathbf{C}^\infty_c(\R^n;\R^m)$. Since
$u_h\in W^{1,p}_{X^h,0}(\Om)$, approximation by functions in
$\mathbf{C}^\infty_c(\Om)$ gives
\begin{equation}\label{auxiliariagvtgbvgfv}
        \int_\Om \langle X^hu_h,\psi\rangle_{\R^m}\,{\rm d}x
    =-\int_\Om u_h\,\divv((\C^h)^T\psi)\,{\rm d}x.
\end{equation}
By \eqref{ublip}, $\divv((\C^h)^T\psi)$ is uniformly bounded in
$L^\infty(\Om)$. Hence, by reflexivity and up to a subsequence, $(u_h\,\divv((\C^h)^T\psi))_h$ is weakly convergent in $L^p(\Om)$.
Let $\eta\in\mathbf{C}^\infty_c(\Om)$. By 
\eqref{a0} and \eqref{quasichiusuradebolehpsuu}, 
\begin{equation*}
\begin{split}
        \lim_{h\to\infty}\int_\Om \eta u_h\,\divv((\C^h)^T\psi)\,{\rm d}x
    &=-\lim_{h\to\infty}\int_\Om \eta\langle X^hu_h,\psi\rangle_{\R^m}\,{\rm d}x
      -\lim_{h\to\infty}\int_\Om u_h\langle X^h\eta,\psi\rangle_{\R^m}\,{\rm d}x\\
    &=-\int_\Om \eta\langle Xu,\psi\rangle_{\R^m}\,{\rm d}x
      -\int_\Om u\langle X\eta,\psi\rangle_{\R^m}\,{\rm d}x\\
    &=\int_\Om \eta u\,\divv(\C^T\psi)\,{\rm d}x.
\end{split}
\end{equation*}
Therefore, passing to the limit \eqref{auxiliariagvtgbvgfv}, we conclude that 
\[
    \int_\Om \langle Xu,\psi\rangle_{\R^m}\,{\rm d}x
    =-\int_\Om u\,\divv(\C^T\psi)\,{\rm d}x.
\]
Since $\psi$ was arbitrary in $\mathbf{C}^\infty_c(\R^n;\R^m)$, this
proves that $X\widetilde u=\widetilde{Xu}$ in $\R^n$. In particular, $\tilde u\in W^{1,p}_X(\rr^n)$. Since $\Om$ has Lipschitz boundary, the thesis then follows applying \cite[Theorem 5]{MR2732591}.
\end{proof}

We point out that the boundary assumption in the second case of \Cref{stranachiusura} cannot be dropped.
\begin{example}
Let $\Om=(-1,1)^2\setminus\{0\}$ and $p=4$. Choose a smooth function $\kappa$
equal to zero on $[-1,1]$ and to one outside $(-2,2)$. Set
\[
    X_1^h=\frac{\partial}{\partial x_1},\qquad X^h_2=x_1\kappa(hx_1)\frac{\partial}{\partial x_2}.
\]
The family $(X^h)_h$ converges to the Grushin vector fields as in \Cref{Section2}. Moreover, $(X^h)_h\in\classdueraff\setminus \classuno$.
Peak $u\in\mathbf C^\infty_c((-1,1)^2)$ and
$v\in\mathbf C^\infty_c((-1,1))$ both equal to one in a neighborhood of $0$, and set 
\[
    w_h(x)=v(hx_1)v(h^5x_2),\qquad u_h=u-w_h.
\]
Then $u_h\in\mathbf C^\infty_c(\Om)$. A simple check shows that $u_h\to u$ and
$X^hu_h\to Xu$ strongly in $L^4((-1,1)^2)$ and $L^4((-1,1)^2;\rr^2)$, respectively.
On the other hand, it is not difficult to verify that 
$|f(0)|\leq C\|Xf\|_{L^4((-1,1)^2;\R^2)}$ for some $C> 0$ and for every $f\in\mathbf C^\infty_c((-1,1)^2)$. In particular, any $\varphi\in \mathbf C^\infty_c(\Om)$ satisfies
\begin{equation*}
    1=u(0)=(u-\varphi)(0)\leq C \|X(u-\varphi)\|_{L^4((-1,1)^2;\R^2)}=C\|X(u-\varphi)\|_{L^4(\Om;\R^2)}.
\end{equation*}
Therefore, $u\notin W^{1,4}_{X,0}(\Om)$.
\end{example}

We can finally characterize affine Sobolev spaces by approximation along $(X^h)_h$. For $\varphi\in W^{1,p}_X(\Om)$, set
\begin{equation*}
\begin{split}
    W^{1,p}_{X,\varphi}&(\Om,(X^h)_h)\coloneqq\{u\in W^{1,p}_X(\Om)\,:\,\text{ there exists }(v_h)_h\subseteq \mathbf{C}^\infty _c(\Om)\text{ such that }\\
    &\quad v_h\to u-\varphi\text{ strongly in }L^p(\Om)\text{ and }X^h v_h\to X(u-\varphi)\text{ strongly in }L^p(\Om;\R^m)\}.
\end{split}
\end{equation*}
\begin{proposition}\label[proposition]{MeySerAffine}
Let $1<p<\infty$. Let $\varphi\in W^{1,p}_X(\Om)$. The following facts hold:
\begin{itemize}
    \item [(i)] if $(X^h)_h$ satisfies \eqref{a0}, then $$W^{1,p}_{X,\varphi}(\Om)\subseteq W^{1,p}_{X,\varphi}(\Om,(X^h)_h);$$
    \item [(ii)] if in addition either $(X^h)_h\in\classuno$ or $(X^h)_h\in\classdue$ and $\Om$ has Lipschitz boundary, then $$W^{1,p}_{X,\varphi}(\Om)= W^{1,p}_{X,\varphi}(\Om,(X^h)_h).$$
\end{itemize}
\end{proposition}
\begin{proof}
Given $u\in W^{1,p}_{X,\varphi}(\Om)$, choose $(v_k)_k\subseteq\mathbf C^\infty_c(\Om)$ converging to $u-\varphi$ in $W^{1,p}_X(\Om)$. For fixed $k$, \eqref{a0} gives $X^hv_k\to Xv_k$ strongly in $L^p(\Om;\R^m)$. A diagonal choice $k(h)\to\infty$ gives the claim. If in addition either $(X^h)_h\in\classuno$ or $(X^h)_h\in\classdue$ and $\Om$ has Lipschitz boundary, the reverse inclusion is a direct consequence of \Cref{stranachiusura}.
\end{proof}
\section{Uniform approximation by smooth functions}\label{Section4}
In this section we establish the counterpart of \Cref{MeySer}, replacing a fixed $X$-gradient with a sequence $(X^h)_h$ satisfying the sole assumption \eqref{a0}. In particular, our results will hold both in $\classuno$ and in $\classdue$. The approximation is local on arbitrary open sets, and global on open sets with Lipschitz boundary.
Fix a sequence $(X^h)_h$ satisfying \eqref{a0}.
In particular, there exists a function $\sigma:[0,\infty)\longrightarrow(0,\infty)$ such that
\begin{equation}\label{modulocontmatrici}
   \sup\{|(\C^h(x)-\C(x))\xi|\,:\,x\in\Om,\,\xi\in\R^n,\,|\xi|=1\}\leq \sigma(h)^{n+2}
\end{equation}
for any $h\in\N$ and   $\lim_{h\to\infty}\sigma(h)=0.$
Exploiting \eqref{modulocontmatrici}, we introduce a particular sequence of mollifiers which is tailored to our setting. Precisely, Let $J\in \mathbf{C}^\infty _c(\R^n)$ be a standard spherically symmetric mollifier. 
We define the family of mollifiers $(J_h)_h$ by letting
\begin{equation}\label{mollificatorigiusti}
    J_h(x)=\frac{1}{\sigma(h)^n}J\left(\frac{x}{\sigma(h)}\right)\qquad\text{for any $x\in \R^n$ and any $h\in\N$.}
\end{equation}
 Clearly, there is a constant $C(J,n)>0$, depending only on $J$ and $n$, such that
\begin{equation}\label{gradmollifgiusto}
    |D J_h(x)|\leq\frac{C(J,n)}{\sigma(h)^{n+1}}\qquad\text{for any $x\in \R^n$ and any $h\in\N$. }
\end{equation}
\subsection{A uniform Meyers-Serrin type approximation}
The following is the technical core of this section.
\begin{lemma}\label{premeyers}
Let $(X^h)_h$ satisfy \eqref{a0}. Let $1\leq p<\infty$, $u\in L^p(\Om)$ and $(J_h)_h$ be as in \eqref{mollificatorigiusti}. Then
    \begin{equation*}
        \lim_{h\to\infty}\|X^h(J_h\ast u)-X(J_h\ast u)\|_{L^p(\Om;\R^m)}=0.
    \end{equation*}
\end{lemma}
\begin{proof}
Notice that, in view of \eqref{modulocontmatrici}, \eqref{gradmollifgiusto} and Jensen's inequality,
\begin{equation*}
\begin{split}
    \|X^h(J_h\ast u)-X&(J_h\ast u)\|^p_{L^p(\Om;\R^m)}=\int_\Om|X^h(J_h\ast u)-X(J_h\ast u)|^p\,{\rm d}x\\
    &=\int_\Om|(\C^h-\C) D(J_h\ast u)|^p\,{\rm d}x\\
    &\overset{\eqref{modulocontmatrici}}{\leq }\sigma(h)^{p(n+2)}\int_\Om|D(J_h\ast u)|^p\,{\rm d}x\\
    &=\sigma(h)^{p(n+2)}\int_\Om|(D J_h\ast u)|^p\,{\rm d}x\\
    &=\sigma(h)^{p(n+2)}\int_\Om\left|\int_{B(x,\sigma(h))}D J_h(x-y)u(y)\,{\rm d}y\right|^p\,{\rm d}x\\
    &=(\omega_n)^p\sigma(h)^{p(2n+2)}\int_\Om\left|\ave_{B(x,\sigma(h))}D J_h(x-y)u(y)\,{\rm d}y\right|^p\,{\rm d}x\\
    &\leq(\omega_n)^{p-1}\sigma(h)^{p(2n+2)-n}\int_\Om\int_{B(x,\sigma(h))}|D J_h(x-y)u(y)|^p\,{\rm d}y\,{\rm d}x\\
    &\overset{\eqref{gradmollifgiusto}}{\leq} C(J,n)^p(\omega_n)^{p-1}\sigma(h)^{p(2n+2)-n-p(n+1)}\int_\Om\int_{B(x,\sigma(h))}|u(y)|^p\,{\rm d}y\,{\rm d}x\\
    &\leq|\Om|\, C(J,n)^p(\omega_n)^{p-1}\sigma(h)^{p(2n+2)-n-p(n+1)}\|u\|^p_{L^p(\Om)},\\
\end{split}
\end{equation*}
where $\omega_n$ is the Lebesgue measure of the Euclidean unit ball of $\R^n$. Since
\begin{equation}\label{numeripen}
    p(2n+2)-n-p(n+1)=n(p-1)+p\geq p>0,
\end{equation}
the thesis follows.
\end{proof}
In view of \Cref{premeyers}, we prove the following Meyers-Serrin type approximation result.
\begin{theorem}\label{meyser}
Let $(X^h)_h$ satisfy \eqref{a0}, let $1\leq p<\infty$ and let $u\in W^{1,p}_X(\Om)$. There exists $(u_h)_h\subseteq \mathbf{C}^\infty(\Om)\cap W^{1,p}_{X^h}(\Om)$ such that $u_h\to u$ strongly in $L^p(\Om)$ and
\begin{equation*}
    X^hu_h\to Xu\quad\text{strongly in $L^p(A;\R^m)$ for every open $A\Subset\Om$.}
\end{equation*}
If $\Om$ has Lipschitz boundary, $(u_h)_h$ can be chosen so that the latter convergence holds on $\Om$.
\end{theorem}
\begin{proof}
Denote by $\widetilde u$ the zero extension of $u$ to $\R^n$ and set
$u_h=J_h*\widetilde u$ on $\Om$, with $J_h$ as in \Cref{premeyers}.
Then $u_h\in\mathbf C^\infty(\Om)\cap W^{1,p}_{X^h}(\Om)$ and
$u_h\to u$ strongly in $L^p(\Om)$. Fix an open set $A\Subset\Om$.
By \cite[Proposition~1.2.2]{MR1437714}, $Xu_h\to Xu$ strongly in
$L^p(A;\R^m)$. By \Cref{premeyers}, we conclude that $X^hu_h\to Xu$ strongly in $L^p(A;\R^m)$.
Assume now that $\Om$ has Lipschitz boundary. By \cite[Theorem~4]{MR2732591}, there exists
$(v_k)_k\subseteq\mathbf C^\infty_c(\R^n)$ whose restriction to $\Om$ converges
to $u$ strongly in $W^{1,p}_X(\Om)$. For each fixed $k$, \eqref{a0} gives
\[
    \|X^hv_k-Xv_k\|_{L^p(\Om;\R^m)}
    \leq \|\C^h-\C\|_{L^\infty(\Om)}
         \|Dv_k\|_{L^p(\Om;\R^n)}
    \longrightarrow 0
    \qquad\text{as }h\to\infty.
\]
By a diagonal argument, we may choose $k(h)\to\infty$ such that
\[
    \|X^hv_{k(h)}-Xv_{k(h)}\|_{L^p(\Om;\R^m)}\to0
    \qquad\text{as }h\to\infty.
\]
Setting $u_h=v_{k(h)}$, we conclude that
$u_h\to u$ strongly in $L^p(\Om)$ and $X^hu_h\to Xu$ strongly in
$L^p(\Om;\R^m)$.
\end{proof}
As in the previous section, the global statement can fail without the regularity assumption on $\partial\Om$.
\begin{example}
Let $Q=(-1,1)^2$. Let $E\subseteq(-1,1)$ be compact, nowhere dense and with positive measure (cf. \cite{MR1681462}). Set $\Om=Q\setminus(\{0\}\times E)$. Then $\Om$ is open and bounded, but $\partial\Om$ is not Lipschitz. Fix $p>2$, and set
\begin{equation*}
  X_1^h=X_1=\frac{\partial}{\partial x_1},\qquad X^h_2=\frac 1 h\frac{\partial}{\partial x_2},\qquad  X_2=0,\qquad 
  u(x_1,x_2)=\mathbf{1}_{(0,1)}(x_1)\mathbf{1}_E(x_2).
\end{equation*}
Then $(X^h)_h\in\classuno\cap\classdue$ and $Xu=0$ in $\Om$, whence $u\in W^{1,p}_X(\Om)$. Moreover, for each $h$, $W^{1,p}_{X^h}(\Om)=W^{1,p}(\Om)$. Since $p>2$ and $E$ is closed and nowhere dense, Morrey's inequality
implies that every $v\in W^{1,p}(\Om)$ extends to a function in $W^{1,p}(Q)$.
Assume by contradiction that there exists $(u_h)_h\subseteq W^{1,p}(\Om)$ such that $u_h\to u$ and $X_1 u_h\to 0$ strongly in $L^p(\Om)$. Then $u_h\to u$ and $X^h_1 u_h\to 0$ strongly in $L^p(Q)$. Choose
$\psi\in\mathbf C^\infty_c(Q)$ equal to $1$ in a neighborhood of $\{0\}\times E$.
Then
\begin{equation*}
    -\int_Q u\,X_1\psi\,{\rm d}x_1\,{\rm d}x_2
    =\lim_{h\to\infty}\int_Q (X_1^hu_h)\psi\,{\rm d}x_1\,{\rm d}x_2
    =0.
\end{equation*}
On the other hand, by the definition of $u$,
\begin{equation*}
    -\int_Q u\,X_1\psi\,{\rm d}x_1\,{\rm d}x_2
    =-\int_E\int_0^1X_1\psi(x_1,x_2)\,{\rm d}x_1\,{\rm d}x_2
    =\int_E\psi(0,x_2)\,{\rm d}x_2
    =|E|>0,
\end{equation*}
a contradiction.
\end{example}
\subsection{A uniform Anzellotti-Giaquinta type approximation}
Although beyond the scope of this paper, for the sake of completeness we provide an Anzellotti-Giaquinta approximation result in the case of \emph{functions of $X$-bounded variation}. We refer to \cite{MR1437714} for their main definitions and properties.
Again, the key of the proof is encoded in the following lemma.
\begin{lemma}\label{preanzgiac}
   Let $(X^h)_h$ satisfy \eqref{a0}. Let $u\in L^1(\Om)$. Let $(J_h)_h$ be as in \eqref{mollificatorigiusti}. Then
    \begin{equation*}
        \lim_{h\to\infty}\sup\left\{\int_\Om u\divv\left(\left((\C^h)^T\varphi\right)\ast J_h-\C^T(\varphi\ast J_h)\right){\rm d}x\,:\,\varphi\in \mathbf{C}^\infty _c(\Om;\R^m),\,\|\varphi\|_\infty\leq 1\right\}=0.
    \end{equation*}
\end{lemma}
\begin{proof}
    Let $\varphi$ be as in the statement, and set
    \begin{equation*}
        S=\int_\Om u\divv\left(\left((\C^h)^T\varphi\right)\ast J_h-\C^T(\varphi\ast J_h)\right){\rm d}x.
    \end{equation*}
    Then
    \begin{equation*}
    \begin{split}
    S
    &=\sum_{i=1}^n\sum_{j=1}^m\int_\Om u\frac{\partial}{\partial x_i} \int_{\R^n}\left(c^h_{j,i}(y)-c_{j,i}(x)\right)\varphi_j(y)J_h(x-y)\,{\rm d}y\,{\rm d}x\\
    &=\sum_{i=1}^n\sum_{j=1}^m\int_\Om u \int_{\R^n}\left(c^h_{j,i}(y)-c_{j,i}(x)\right)\varphi_j(y)\frac{\partial J_h}{\partial x_i}(x-y)\,{\rm d}y\,{\rm d}x\\
    &\quad-\sum_{i=1}^n\sum_{j=1}^m\int_\Om u \int_{\R^n}\frac{\partial c_{j,i}}{\partial x_i}(x)\varphi_j(y)J_h(x-y)\,{\rm d}y\,{\rm d}x\\
     &=\sum_{i=1}^n\sum_{j=1}^m\int_\Om u \int_{\R^n}\left(c^h_{j,i}(y)-c_{j,i}(y)\right)\varphi_j(y)\frac{\partial J_h}{\partial x_i}(x-y)\,{\rm d}y\,{\rm d}x\\
     &\quad+\sum_{i=1}^n\sum_{j=1}^m\int_\Om u \int_{\R^n}\left(c_{j,i}(y)-c_{j,i}(x)\right)\varphi_j(y)\frac{\partial J_h}{\partial x_i}(x-y)\,{\rm d}y\,{\rm d}x\\
     &\quad-\sum_{i=1}^n\sum_{j=1}^m\int_\Om u \int_{\R^n}\frac{\partial c_{j,i}}{\partial x_i}(x)\varphi_j(y)J_h(x-y)\,{\rm d}y\,{\rm d}x.\\
     \end{split}
    \end{equation*}
Denote respectively by $\mathrm{I},\mathrm{II}$ and $\mathrm{III}$ the last three terms. Arguing as in \cite[Lemma 2.1.1]{MR1437714}, $|\mathrm{II}+\mathrm{III}|\to 0$ as $h\to+\infty$ uniformly with respect to $\varphi$. We are left to estimate $I$. In view of \eqref{modulocontmatrici} and \eqref{gradmollifgiusto}, we get that
\begin{equation*}
\begin{split}
    |\mathrm{I}|&
    \overset{\eqref{modulocontmatrici}}{\leq }\sigma(h)^{n+2}\sum_{i=1}^n\sum_{j=1}^m\int_\Om |u| \int_{\R^n}\left|\varphi_j(y)\frac{\partial J_h}{\partial x_i}(x-y)\right|\,{\rm d}y\,{\rm d}x\\
    &\overset{\eqref{gradmollifgiusto}}{\leq}C(J,n)\sigma(h)\sum_{i=1}^n\sum_{j=1}^m\int_\Om |u| \int_{\R^n}|\varphi_j(y)|\,{\rm d}y\,{\rm d}x\\
    &\leq |\Om\,|m\, n\, C(J,n)\,\sigma(h)\int_\Om |u| \,{\rm d}x. \end{split}
    \end{equation*}
    Therefore $\mathrm{I}$ goes to $0$ as $h\to+\infty$ uniformly with respect to $\varphi$, whence the thesis follows.
\end{proof}
The anisotropic Anzellotti-Giaquinta approximation result reads as follows.
\begin{theorem}\label{thmanzellottigiaqervgr}
Let $(X^h)_h$ satisfy \eqref{a0} and let $u\in BV_X(\Om)$. There exists $(u_h)_h\subseteq \mathbf{C}^\infty(\Om)\cap W^{1,1}_{X^h}(\Om)$ such that $u_h\to u$ strongly in $L^1(\Om)$ and
\begin{equation*}
    \lim_{h\to\infty}\int_A|X^hu_h|\,{\rm d}x=|Xu|(A)
\end{equation*}
for every open $A\Subset\Om$ such that $|Xu|(\partial A)=0$. If $\Om$ has Lipschitz boundary, the sequence can be chosen so that the same equality holds also for $A=\Om$.
\end{theorem}
\begin{proof}
The local assertion follows from \Cref{preanzgiac} as in the proof of \cite[Proposition~2.2.1]{MR1437714}.
If $\Om$ has Lipschitz boundary, the global assertion follows by applying \Cref{meyser} to each element of the approximating sequence $(u_h)_h\subseteq W^{1,1}_X(\Om)$ given by \cite[Theorem~2.2.2]{MR1437714} and by a standard diagonal argument.
\end{proof}

\section{Uniform compactness and Poincaré inequalities}\label{Section5}
In this section we prove the counterparts of \Cref{relkon} and \Cref{poincaprelprop} for moving anisotropies.
We consider either sequences in $\classunoraff$ or in $\classdueraff$ when $\Om$ has Lipschitz boundary.  Before going on, we point out that neither the assumptions of $\classuno$ nor those of $\classdue$ are strong enough to guarantee compactness, as the next example shows.
\begin{example}\label{exseq}
    For any $h\in\N$, consider the family $\X^h=(X_1,X^h_2)$ defined on $\R_{(x_1,x_2)}^2$ by
    \begin{equation*}
        X_1=\frac{\partial}{\partial x_1}\qquad \text{and}\qquad X^h_2=\frac{1}{h}\frac{\partial}{\partial x_2}.
    \end{equation*}
    For any $h\in\N$, $X_1$ and $X_2^h$ are smooth and Lipschitz continuous.
    Moreover, $(\X^h)_h$ converges uniformly, with all its derivatives, to $\X=(X_1,0)$, whence $(\X^h)_h\in\classuno\cap\classdue$.
    Let $\Om=(-\pi,\pi)^2$. For any $h\in\N$, set $\varphi_h=0$, $\varphi=0$ and $u_h(x_1,x_2)\coloneqq\cos(x_1/2)\sin(hx_2)$ for any $(x_1,x_2)\in\R^2$.
    Then $(u_h)_h\subseteq\mathbf{C}^\infty(\overline\Om)$ and, since $u_h=0$ on $\partial\Om$, we have $u_h\in W^{1,2}_0(\Om)\subseteq W^{1,2}_{\X^h,0}(\Om)$ for any $h\in\N$.
    Moreover,
    \begin{equation*}
        \X^h u_h(x_1,x_2)
        =\left(-\frac{1}{2}\sin(x_1/2)\sin(hx_2),\,\cos(x_1/2)\cos(hx_2)\right)
    \end{equation*}
    for any $(x_1,x_2)\in\R^2$. Then, $\|u_h\|_{L^2(\Om)}=\pi$ and $\X^h u_h$ is bounded in ${L^2(\Om;\R^2)}$.
    By the Riemann-Lebesgue lemma, $u_h\to 0$ weakly in $L^2(\Om)$, but $(u_h)_h$ does not admit any subsequence converging strongly in $L^2(\Om)$.
\end{example}
Our first result is the following uniform Rellich-Kondrachov compactness property.
\begin{theorem}[Uniform Rellich-Kondrachov]\label{nuovothmcompattezza}
    Assume that either $(X^h)_h\in\classunoraff$ or $(X^h)_h\in\classdueraff$ and $\Om$ has Lipschitz boundary. Let $1<p<\infty$.
    Let $(\varphi_h)_h\subseteq W^{1,p}_{X^h}(\Om)$, and let $\varphi\in W^{1,p}_{X}(\Om)$. Assume that
    \begin{equation}\label{quasichiusuradebolehpmapercomp}
        \varphi_h\to\varphi\text{ strongly in }L^p(\Om)\qquad\text{and}\qquad X^h\varphi_h\to X\varphi \text{ weakly in }L^p(\Om;\R^m).
    \end{equation}
    Let $(u_h)_h\subseteq W^{1,p}_{X^h,\varphi_h} (\Omega)$.
Assume that there exists $M>0$ such that
\begin{equation}\label{boundedhpincompnew}
    \|u_h\|_{L^p(\Om)},\|X^hu_h\|_{L^p(\Om;\R^m)}\leq M\text{ for any $h\in\N$.}
\end{equation}
 Then there exists $u\in W^{1,p}_{X,\varphi}(\Om)$ such that, up to a subsequence, $u_h\to u$ strongly in $L^p(\Om)$.
\end{theorem}
We first establish an interior compactness property when
$(X^h)_h\in\classdueraff$.
\begin{lemma}\label[lemma]{convimmin348598686}
Let $U\subseteq\R^n$ be a bounded open set. Assume that
$(X^h)_h\in\classdueraff$ on $U$. Let $1<s<\infty$ and suppose that
$(u_h)_h$ and $(X^hu_h)_h$ are bounded in $L^s_{\rm loc}(U)$ and
$L^s_{\rm loc}(U;\R^m)$, respectively. Then there exists $u\in L^s_{\rm loc}(U)$ such that, up to a
subsequence, $u_h\to u$ almost everywhere in $U$.
\end{lemma}
\begin{proof}
Fix a compact set $K\subseteq U$ and a smooth cut-off function $\chi$ between $K$ and $U$. Fix a Euclidean ball $B$ such that $\overline U\subseteq B$. Multiplying $\C^h$ and $\C$ by $\chi$ allows to extend both matrices to $B$. The new matrices, still denoted by $\C^h$ and $\C$, satisfy \eqref{a0} and \eqref{ublip} on $B$. Moreover, by construction, the new family $(X^h)_h$ still belongs to $\classdueraff$ in a neighborhood of $K$. \\
\noindent \textbf{Step 1.}
First, we show that  an arbitrary sequence of measurable functions $(v_h)_h$ satisfying
\begin{equation}\label{successionerellich3848686}
 \sup_h\|v_h\|_{L^\theta(B)}<\infty\quad\text{for every }1<\theta<\infty,\qquad
 \sup_h\|X^hv_h\|_{L^s(B;\R^m)}<\infty\quad\text{for some }s>1
\end{equation}
 is relatively compact in $L^2(K)$. Fix $\theta$. For $z\in L^\theta(B)$, let $Pz\in W^{1,\theta}_0(B)$ be the unique solution to 
 \begin{equation*}
     -\Delta (Pz)=z\quad\text{on B},\qquad Pz=0\quad\text{on }\partial B.
 \end{equation*}
It is well-known \cite[Theorem~9.15 and Lemma~9.17]{GT} that $Pz\in W^{2,\theta}(B)$, and that
\begin{equation}\label{stimasobolevlapl}
    \|Pz\|_{W^{2,\theta}(B)}\leq C_{\theta,B}\|z\|_{L^\theta(B)}
\end{equation}
for some $C_{\theta,B}>0$. By a diagonal argument, up to a subsequence, \eqref{successionerellich3848686} implies that there exists $v\in L^\theta(B)$, for every $1\les\theta\les\infty$, such that $v_h\to v$ weakly in $L^\theta(B)$.  Set $q_h=DPv_h$
and $q=DPv$. By \eqref{stimasobolevlapl} and the weak convergence of $(v_h)_h$, $q_h\to q$ weakly in every $W^{1,\theta}(B;\R^n)$.
Since by construction $v_h=-\divv q_h$ and $v=-\divv q$, to conclude the claim
it suffices to prove that $Dq_h\to Dq$ strongly in $L^2(K;\R^{n\times n})$.\\
\noindent \textbf{Step 2.} 
For any $h$ and any $j=1,\ldots,m$, denote by $c_j^h$ and $c_j$ the $j$-th row of $\C^h$ and $\C$, respectively. We write $X_j^h q_h$ for the vector in $\R^n$ whose $i$-th component is
\begin{equation}\label{definizioneincoordinate}
     (X_j^h q_h)_i=X_j^h q_{h,i}
 =\sum_{\ell=1}^n c_{j,\ell}^h\,\partial_\ell q_{h,i},
 \qquad i=1,\ldots,n.
\end{equation}
A simple computation reveals that, distributionally, 
\begin{align*}
 \divv(X_j^h q_h)
 &=-X_j^h v_h+\sum_{i,\ell=1}^n
       (\partial_i c_{j,\ell}^h)\partial_\ell q_{h,i},\\
 \partial_k(X_j^h q_{h,i})-\partial_i(X_j^h q_{h,k})
 &=\sum_{\ell=1}^n\big((\partial_k c_{j,\ell}^h)\partial_\ell q_{h,i}
                  -(\partial_i c_{j,\ell}^h)\partial_\ell q_{h,k}\big),
 \qquad i,k=1,\ldots,n .
\end{align*}
Both right-hand sides are bounded in $L^s(B)$ uniformly in $h$.
Indeed, \eqref{successionerellich3848686} bounds $X_j^h v_h$ in $L^s(B)$,
\eqref{ublip} bounds $Dc_j^h$ in $L^\infty(B,\rr^{n\times n})$, and the Dirichlet
estimate bounds $Dq_h=D^2Pv_h$ in $L^s(B;\rr^{n\times n})$.
We show that these bounds control all first derivatives of $X_j^h q_h$.
Let $F\in L^s(B;\R^n)$ have compact support in $B$, and suppose that
$\divv F$ and $\partial_kF_i-\partial_iF_k$ belong to $L^s(B)$
for all $i,k$. Distributionally,
\begin{equation}\label{identitadistribuzlaplncierfjnefi}
     \Delta F_i=\partial_i\divv F
       +\sum_{k=1}^n\partial_k(\partial_kF_i-\partial_iF_k).
\end{equation}
Fix $\psi\in\mathbf C^\infty_c(B)$ and choose
$\eta\in\mathbf C^\infty_c(B)$ equal to one on a neighborhood of
$\operatorname{supp}F$. By elliptic interior regularity, $P\psi$ is smooth in
$B$, so $\eta\,\partial_\ell P\psi\in\mathbf C^\infty_c(B)$. Since
$-\Delta P\psi=\psi$ and $\eta=1$ near $\operatorname{supp}F$, \eqref{identitadistribuzlaplncierfjnefi} yields
\begin{align*}
 \langle\partial_\ell F_i,\psi\rangle
 &
   =\int_B F_i\,\Delta(\eta\,\partial_\ell P\psi)\,{\rm d}x\\
 &=-\int_B\divv F\,\partial_i(\eta\,\partial_\ell P\psi)\,{\rm d}x
   -\sum_{k=1}^n\int_B(\partial_kF_i-\partial_iF_k)
                   \,\partial_k(\eta\,\partial_\ell P\psi)\,{\rm d}x\\
 &=-\int_B\divv F\,\partial_{i\ell}P\psi\,{\rm d}x
   -\sum_{k=1}^n\int_B(\partial_kF_i-\partial_iF_k)
                   \,\partial_{k\ell}P\psi\,{\rm d}x .
\end{align*}
 Hence, H\"older's inequality and \eqref{stimasobolevlapl} grant that
$\partial_\ell F_i\in L^s(B)$ for every $i,\ell$. Moreover, for some $C_{s,B}>0$,
\[
 \|DF\|_s\leq C_{s,B}\Big(\|\divv F\|_s+
       \sum_{i,k=1}^n\|\partial_kF_i-\partial_iF_k\|_s\Big).
\]
Fix $j$. Since $X_j^h q_h$ is supported in
$\operatorname{supp}\chi\Subset B$, we may apply the above estimate
with $F=X_j^h q_h$. The bounds on its divergence and on the differences
$\partial_k(X_j^h q_{h,i})-\partial_i(X_j^h q_{h,k})$ give
$\|D(X_j^h q_h)\|_s\leq C$ uniformly in $h$. Moreover,
$\|X_j^h q_h\|_s\leq\|c_j^h\|_\infty\|Dq_h\|_s\leq C$, because of \eqref{ublip} and \eqref{stimasobolevlapl}.
Hence $(X_j^h q_h)_h$ is bounded in $W^{1,s}(B;\R^n)$. Therefore, by the
Euclidean Rellich theorem $(X_j^h q_h)_h$ is relatively compact in $L^s(B;\R^n)$.\\
\noindent \textbf{Step 3.} 
By \eqref{a0} and the uniform
bound on $\|Dq_h\|_s$,
\begin{equation}\label{erijnverivjntrvijnrfvi}
     \|X_jq_h-X_j^h q_h\|_s
 \leq \|c_j-c_j^h\|_\infty\|Dq_h\|_s \longrightarrow 0 .
\end{equation}
In particular,
$(X_jq_h)_h$ is also relatively compact in
$L^s(B;\R^n)$. Since
$Dq_h\to Dq$ weakly in $L^s$, then $X_jq_h\to X_jq$ weakly 
in $L^s$. This and the relative compactness of $(X_jq_h)_h$ imply that $X_jq_h\to X_jq$ strongly in $L^s(B;\R^n)$. Therefore, by \eqref{erijnverivjntrvijnrfvi}, $X_j^h q_h\to X_j q$ strongly in $L^s(B;\R^n)$.
If $1<\theta<s$, clearly $X_jq_h$ to $X_jq$
strongly in $L^\theta(B;\R^n)$.
If instead $\theta>s$, $(Dq_h)_h$ is bounded in $L^{2\theta}$ by \eqref{stimasobolevlapl}. Hence $(X_jq_h-X_jq)_h$ is bounded in $L^{2\theta}$
and tends to zero in $L^s$. 
Choose $\vartheta\in(0,1)$ such that
\[
    \frac{1}{\theta}
    =\frac{\vartheta}{s}+\frac{1-\vartheta}{2\theta}.
\]
H\"older's
inequality gives
\[
    \|X_jq_h-X_jq\|_{\theta}
    \leq \|X_jq_h-X_jq\|_s^{\vartheta}
         \|X_jq_h-X_jq\|_{2\theta}^{1-\vartheta}
    \leq C\|X_jq_h-X_jq\|_s^{\vartheta}
    \longrightarrow 0 .
\] 
Therefore, we conclude that
\begin{equation}\label{tesidelpassoindumabase}
    X_jq_h\to X_jq
 \qquad\text{strongly in }L^\theta(B;\R^n)
 \quad\text{for every }1<\theta<\infty .
\end{equation}
\noindent \textbf{Step 4.} 
Let $Y$ be a smooth vector field on $B$. We write $Y=b\cdot D$ to say that $Y$ is represented by the vector-valued map $b$. Assume that $b$ and
$Db$ are bounded.
Suppose that, whenever
$(z_h)_h$ satisfies \eqref{successionerellich3848686} and converges weakly
to $z$ in every $L^\theta(B)$, $1<\theta<\infty$, then
\begin{equation}\label{conseguenzadiipotesiinduttiva}
     Y(D Pz_h)\to Y(DPz)
 \qquad\text{strongly in }L^\theta(B;\R^n)
 \quad\text{for every }1<\theta<\infty .
\end{equation}
Here $Y(DPz_h)$ and $Y(DPz)$ are defined in coordinates as in \eqref{definizioneincoordinate}, namely 
\begin{equation*}
    Y(DPz_h)=D^2Pz_h\,b,\qquad  Y(DPz)=D^2Pz\,b.
\end{equation*}
Step~3 gives this property whenever $Y=X_1,\ldots,X_m$. Indeed, notice that, when $z=v$ and $z_h=v_h$, then 
\begin{equation*}
   D^2Pz_h\,b=Yq_h,\qquad D^2Pz\,b=Yq.
\end{equation*}
Fix $1<\theta<\infty$.
We claim that, if $(\eta_h)_h\subseteq W^{1,\theta'}_0(B)$ is any sequence such that $\eta_h\to\eta$ weakly in $W^{1,\theta'}_0(B)$,
then, for every $(z_h)_h$ as above,
\begin{equation}\label{compactadjoint}
 \int_B z_h\,Y\eta_h\,{\rm d}x
 \longrightarrow \int_B z\,Y\eta\,{\rm d}x .
\end{equation}
Indeed, since $\Delta Pz_h=-z_h$, then distributionally
\[
 \divv( D^2Pz_h\,b)
 =\sum_{i,\ell=1}^n(\partial_i b_\ell)\partial_{\ell i}Pz_h
   +Y(\Delta Pz_h)
 =\sum_{i,\ell=1}^n(\partial_i b_\ell)\partial_{\ell i}Pz_h-Yz_h .
\]
Since $\divv(bz_h)=(\divv b)z_h+Yz_h$, and using again $\Delta Pz_h=-z_h$,
\begin{equation*}
 \divv(bz_h)+\divv( D^2Pz_h\,b)
 =
   (\divv b)z_h+
   \sum_{i,\ell=1}^n(\partial_i b_\ell)\partial_{\ell i}Pz_h.
\end{equation*}
Evaluating the above identity at $\eta_h$, which
is allowed by density of $\mathbf C^\infty_c(B)$ in $W^{1,\theta'}_0(B)$, gives
\begin{equation*}
 \int_B z_h\,Y\eta_h\,{\rm d}x
=-\int_B (D^2Pz_h\,b)\cdot D\eta_h\,{\rm d}x-\int_B\eta_h\left((\divv b)z_h+
       \sum_{i,\ell=1}^n(\partial_i b_\ell)\partial_{\ell i}Pz_h
                       \right){\rm d}x .
\end{equation*}
In the first integral on the right-hand side, $D^2Pz_h\,b\to D^2Pz\,b$ strongly
in $L^\theta$ by \eqref{conseguenzadiipotesiinduttiva}, while $D\eta_h\to D\eta$ 
weakly in $L^{\theta'}$. In the second integral, the Euclidean
Rellich theorem gives $\eta_h\to\eta$ strongly in $L^{\theta'}$.
The factor in parentheses converges weakly in $L^\theta$ to the
same expression with $z$ in place of $z_h$, since
$z_h\to z$ weakly and recalling \eqref{stimasobolevlapl}. Passing to the limit,
and using the same identity for $z$ and $\eta$, proves
\eqref{compactadjoint}.\\
\noindent\textbf{Step 5.} 
Let $Y=b\cdot D$ and assume that, for every sequence $(z_h)_h$
satisfying \eqref{successionerellich3848686} and converging weakly to $z$
in $L^\theta(B)$ for every $1<\theta<\infty$, one has
\[
Y(DPz_h)\to Y(DPz)
 \qquad\text{strongly in }L^\theta(B;\R^n)
 \quad\text{for every }1<\theta<\infty .
\]
We prove that, for every such sequence and every $j=1,\ldots,m$,
\begin{equation}\label{tesidelpassoindu}
     [X_j,Y](DPz_h)\to [X_j,Y](DPz)
 \qquad\text{strongly in }L^\theta(B;\R^n)
 \quad\text{for every }1<\theta<\infty .
\end{equation}
Set $Q_h=DPz_h$ and $Q=DPz$.
Fix $i=1,\ldots,n$ and $j=1,\ldots,m$, and put $w_{h,i}^j=[X_j,Y]Q_{h,i}$ and $w_i^j=[X_j,Y]Q_i$. Notice that $w_{h,i}^j\to w_i^j$ weakly in $L^\theta(B)$, for every $1\les\theta\les \infty$.
Moreover, $w_{h,i}^j$ and $(\partial_kQ_{h,i})w_{h,i}^j$ satisfy
\eqref{successionerellich3848686}. Indeed by \eqref{stimasobolevlapl},
$(\partial_kQ_{h,i})_h$ is bounded in $L^\theta(B)$ for every
$1<\theta<\infty$. Moreover, distributionally,
\[
 X_j^h\partial_kQ_{h,i}
 =\partial_k(X_j^h Q_{h,i})-(\partial_kc_j^h)\cdot DQ_{h,i},
 \qquad j=1,\ldots,m.
\]
Arguing as in Step~2, the first term on the right hand side is bounded in $L^s(B)$.
The second has the same bound by \eqref{ublip} and the bound on
$DQ_h$ in $L^s(B;\R^{n\times n})$, which again holds as in Step~2. Thus $(\partial_kQ_{h,i})_h$
satisfies \eqref{successionerellich3848686}.
Suppose that $(f_h)_h$ and $(g_h)_h$ satisfy
\eqref{successionerellich3848686} with a common exponent $s$. By H\"older's inequality,
$$\|f_hg_h\|_\theta\leq\|f_h\|_{2\theta}\|g_h\|_{2\theta}.$$
Moreover, fix $1<t<s$, and take $r$ such that $1/r=1/t-1/s$. It is easy to show that 
\[
 \|X^h(f_hg_h)\|_t
 \leq \|f_h\|_r\|X^hg_h\|_s+\|g_h\|_r\|X^hf_h\|_s .
\]
Since the right-hand side is uniformly bounded, $(f_hg_h)_h$ satisfies \eqref{successionerellich3848686} with exponent $t$ instead of $s$. Moreover, if $\phi\in\mathbf C^\infty(B)$ is such that $\phi$ and $D\phi$ are bounded, and 
if $(f_h)_h$ satisfies \eqref{successionerellich3848686}, then
$(\phi f_h)_h$ satisfies \eqref{successionerellich3848686} with the same $s$. Therefore, $w_{h,i}^j$ and $(\partial_kQ_{h,i})w_{h,i}^j$ satisfy
\eqref{successionerellich3848686}. Since $w_{h,i}^j\to w_i^j$ in $L^2(B)$ weakly, strong convergence in $L^2(B)$ follows from convergence of the $L^2(B)$-norms. Since $(w_{h,i}^j)_h$ is bounded in every finite
$L^\theta(B)$, an argument as in step 3 grants \eqref{tesidelpassoindu}. 
Notice that
\begin{equation*}
 \int_B [X_j^h,Y]Q_{h,i}\,w_{h,i}^j\,{\rm d}x=-\int_B YQ_{h,i}\,\divv(c_j^h w_{h,i}^j)\,{\rm d}x
 -\int_B w_{h,i}^j\,Y(X_j^h Q_{h,i})\,{\rm d}x .
\end{equation*}
%
We now pass to the limit in the two integrals. In the first one,
\[
 \divv(c_j^h w_{h,i}^j)
 =X_j^h w_{h,i}^j+(\divv c_j^h)w_{h,i}^j
\]
is bounded in $L^t(B)$ for some $t>1$, by
\eqref{successionerellich3848686} and \eqref{ublip}. Its distributional
limit is $\divv(c_j w_i^j)$, since $c_j^h\to c_j$ uniformly and
$w_{h,i}^j\to w_i^j$ weakly in $L^t(B)$. But then, boundedness in $L^t(B)$ implies that the convergence
is weak in $L^t(B)$. By the induction hypothesis,
$YQ_{h,i}\to YQ_i$ strongly in $L^{t/(t-1)}(B)$. In conclusion,
\[
 \int_B YQ_{h,i}\,\divv(c_j^h w_{h,i}^j)\,{\rm d}x
 \longrightarrow
 \int_B YQ_i\,\divv(c_j w_i^j)\,{\rm d}x .
\]
For the second integral, notice that, arguing as in Step~2 and~3 and by construction, $X_j^h Q_{h,i}\to X_j Q_{i}$ weakly in $W^{1,s}_0(B)$.
Applying \eqref{compactadjoint} with $z_h=w_{h,i}^j$ and $\eta_h=X_j^h Q_{h,i}$ then gives
\[
 \int_B w_{h,i}^j\,Y(X_j^h Q_{h,i})\,{\rm d}x
 \longrightarrow
 \int_B w_i^j\,Y(X_jQ_i)\,{\rm d}x .
\]
Therefore, we conclude that 
\begin{equation*}
\lim_{h\to\infty}\int_B [X_j^h,Y]Q_{h,i}\,w_{h,i}^j\,{\rm d}x
   =\int_B [X_j,Y]Q_i\,w_i^j\,{\rm d}x
   =\int_B (w_i^j)^2\,{\rm d}x .
\end{equation*}
%
Finally, combining \eqref{a0}, \eqref{ublip} and
\eqref{compactadjoint}, it is easy to conclude that 
\[
\lim_{h\to\infty}
\int_B w_{h,i}^j\bigl([X_j^h,Y]-[X_j,Y]\bigr)Q_{h,i}\,{\rm d}x
=0.
\]
Therefore, by the above discussion, \eqref{tesidelpassoindu} follows.\\
\noindent \textbf{Step 6.} 
Since $X$ satisfies the H\"ormander condition, \eqref{tesidelpassoindumabase} and \eqref{tesidelpassoindu} imply that
\[
    Yq_h\longrightarrow Yq
    \qquad\text{strongly in }L^2(K;\R^n)
\]
for any smooth vector field $Y$. In particular, $Dq_h\to Dq$ strongly in $L^2(K;\R^{n\times n})$. By Step 1, this suffices to prove that the original sequence $(v_h)_h$ is relatively compact in $L^2(K)$. \\
\noindent \textbf{Step 7.}
Finally, let $(u_h)_h$ be as in the statement. For $k\in\N$, let $v_h^k=\chi\max\{-k,\min\{u_h,k\}\}$. Then $(v_h^k)_h$ satisfies \eqref{successionerellich3848686}. By the previous steps, it is relatively compact in $L^2(K)$, whence in $L^1(K)$. Noticing that
\[
 \|u_h-v_h^k\|_{L^1(K)}
 \leq k^{1-s}\|u_h\|_{L^s(K)}^s,
\]
we conclude that $(u_h)_h$ is relatively compact in $L^1(K)$. A diagonal argument exhausting $U$ by compact sets $K$ allows to conclude the proof.
\end{proof}
In order to upgrade a.e. convergence to strong $L^p$-convergence, we need to estimate the integral of $|u_h|^p$ on sets where
$|u_h|$ is large. To this aim, for $g\in L^1_{\rm loc}(\R^n)$, set
\[
    A_rg(x)=\frac{1}{|B(x,r)|}\int_{B(x,r)}g(y)\,{\rm d}y,
    \qquad
    Mg(x)=\sup_{r>0}A_r|g|(x).
\]
Here $Mg$ is the well-known
Hardy-Littlewood \emph{maximal function} associated with $g$.

\begin{lemma}
Let $1<p<\infty$, let $v\in\mathbf C^\infty_c(\R^n)$ and set
$V=(M(|v|^p))^{1/p}$. Let $X$ be a family of $m$ Lipschitz vector
fields on $\R^n$. Then $V$ is Lipschitz and, for every $L>0$,
\begin{align}
    L^p\bigl(|\{V>L\}|+|\{|XV|>L\}|\bigr)
    &\leq C\int_{\R^n}\bigl(|v|^p+|Xv|^p\bigr)\,{\rm d}x,
    \label{eq_rellichneovjnrvojngrvjrn}\\
    \int_{\{|v|>L\}}|v|^p\,{\rm d}x
    &\leq C_n L^p|\{V>L\}|.
    \label{coda3848586885}
\end{align}
Here $C$ depends only on $n,m,p$ and the Lipschitz bound of $\C$,
while $C_n$ depends only on $n$.
\end{lemma}

\begin{proof}
First, \eqref{coda3848586885} follows directly by \cite[formula~(2.15)]{MR1800316}. Moreover, $V$ is Lipschitz, because
\[
    |V(x+h)-V(x)|
    \leq \sup_{r>0}
    \left(\frac{1}{|B(x,r)|}\int_{B(x,r)}
    |v(y+h)-v(y)|^p\,{\rm d}y\right)^{1/p}
    \leq \|Dv\|_\infty |h|\qquad x,h\in\rr^n.
\]
We claim that, almost everywhere in $\R^n$,
\begin{equation}\label{sitmanellemmadueprima}
        |XV|\leq C\bigl(V+(M(|Xv|^p))^{1/p}\bigr).
\end{equation}
We may assume that $v\neq0$, so that $V>0$. Set $g=|v|^p$,
and fix a differentiability point $x$ of both $V$ and $|v|$. Since $g$ is continuous, $A_rg(x)\to g(x)$
as $r\to0$ and $r\mapsto A_r g(x)$ is continuous. Since $g\in L^1(\rr^n)$, $A_rg(x)\to0$ as $r\to\infty$. Therefore, setting $A_0 g(x)=g(x)$, there exists $r_x\geq 0$ such that  $V(x)^p=A_{r_x} g(x)$. If $V(x)=|v(x)|$,
then $DV(x)=D|v|(x)$, because $V\geq |v|$. In this case, \eqref{sitmanellemmadueprima} follows. Assume instead $V(x)\ges |v(x)|$. Then $r_x\ges 0$.
Since $V^p(x)=A_{r_x}g(x)$ and $V^p(y)\geq A_{r_x}g(y)$, then
$D(V^p)(x)=D(A_{r_x}g)(x)$. Moreover, since $r_x$ is an interior maximum point of $r\mapsto A_r g(x)$, then taking the derivative at $r_x$ yields
\begin{equation}\label{cerovjntrvijtrnvijgtr}
        \frac{1}{|\partial B(x,r_x)|}\int_{\partial B(x,r_x)}g\,{\rm d}\mathcal H^{n-1}
    =A_{r_x}g(x)=V(x)^p.
\end{equation}
For $j=1,\ldots,m$, let $c_j$ be the $j$-th row of $\C$. Let
$\nu$ be the outer unit normal to $B(x,r_x)$. Then
\begin{equation*}
    X_j(V^p)(x)=\frac{1}{|B(x,r_x)|}\int_{B(x,r_x)}c_j(x)\cdot Dg (y)\,{\rm d}y=
   \frac{1}{|B(x,r_x)|}\int_{\partial B(x,r_x)}
       g(y)c_j(x)\cdot\nu(y)\,{\rm d}\mathcal H^{n-1}(y).
\end{equation*}
On the other hand, again by the divergence theorem,
\begin{equation*}
    A_{r_x}\bigl(X_jg+g\,\divv c_j\bigr)(x)=\frac{1}{|B(x,r_x)|}\int_{\partial B(x,r_x)}
       g(y)c_j(y)\cdot\nu(y)\,{\rm d}\mathcal H^{n-1}(y).
\end{equation*}
Therefore,
\begin{equation}\label{espressioneduedivvvvvvv}
    X_j(V^p)(x)=  A_{r_x}\bigl(X_jg+g\,\divv c_j\bigr)(x)+  \frac{1}{|B(x,r_x)|}\int_{\partial B(x,r_x)}
       g(y)\big(c_j(x)-c_j(y)\big)\cdot\nu(y)\,{\rm d}\mathcal H^{n-1}(y).
\end{equation}
Since $c_j$ is Lipschitz and  $r_x|\partial B(x,r_x)|=n|B(x,r_x)|$, \eqref{cerovjntrvijtrnvijgtr}, \eqref{espressioneduedivvvvvvv} and H\"older's inequality imply \eqref{sitmanellemmadueprima}.
Combining \eqref{sitmanellemmadueprima} with \cite[Theorem 2.5]{MR1800316}, \eqref{eq_rellichneovjnrvojngrvjrn} follows.
\end{proof}
\begin{lemma}\label[lemma]{rellilocale3456789}
Let $U\subseteq\R^n$ be a bounded open set. Assume that
$(X^h)_h\in\classdueraff$ on $U$. Let $1<p<\infty$ and suppose that
$(u_h)_h$ and $(X^hu_h)_h$ are bounded in $L^p_{\rm loc}(U)$ and
$L^p_{\rm loc}(U;\R^m)$, respectively. Then
$(u_h)_h$ is relatively compact in $L^p_{\rm loc}(U)$.
\end{lemma}
\begin{proof}
Fix an arbitrary compact set $K\subseteq U$. It suffices to prove relative compactness in $K$. The thesis then follows by a standard diagonal argument.
Let $\chi$ be a smooth cutoff function between $K$ and $U$. Fix an open set $U'$ such
that $K\Subset\operatorname{supp}\chi\Subset U'\Subset U$. Without loss of generality, we may assume that $(u_h)_h\subseteq C^\infty_c(U')$ and $\operatorname{supp} u_h\Subset\supp \chi$. Indeed, for each $h$,
\Cref{MeySer} gives $v_h\in\mathbf C^\infty(U')$ such that
\[
    \|v_h-u_h\|_{L^p(U')}
    +\|X^hv_h-X^hu_h\|_{L^p(U';\R^m)}\leq\frac1h.
\]
Notice that 
\[
    X^h(\chi v_h-\chi u_h)
    =\chi(X^hv_h-X^hu_h)+(v_h-u_h)\C^h D\chi
    \qquad\text{in }U'.
\]
Therefore, \eqref{a0} implies that there exists $C\ges 0$ such that 
\[
    \|\chi v_h-\chi u_h\|_{L^p(U)}
    +\|X^h(\chi v_h)-X^h(\chi u_h)\|_{L^p(U;\R^m)}
    \leq \frac{C}{h}.
\]
Similarly, $\chi u_h$ and $X^h(\chi u_h)$ are bounded in $L^p(U)$,
uniformly in $h$. Therefore, $\chi v_h$ and $X^h(\chi v_h)$ are bounded in $L^p(U)$,
uniformly in $h$. Moreover, since $\chi=1$ on $K$, then
$\|\chi v_h-u_h\|_{L^p(K)}\to0$. Hence, it suffices to prove
relative compactness in $L^p(K)$ for $(\chi v_h)_h$. Therefore, assume that $(u_h)_h\subseteq C^\infty_c(U')$ and $\operatorname{supp} u_h\Subset\supp \chi$.
By \Cref{convimmin348598686}, up to a subsequence (and recalling the localization to $U'$), there exists $u\in L^p(U)$ such that 
$u_h\to u$ almost everywhere in $U$.
Extend each $\C^h$ from $U$ to $\R^n$ preserving its uniform Lipschitz bounds as in \eqref{ublip}. Since $u_h\in \mathbf C^\infty_c(U')$, such extensions
preserve the uniform bounds on $u_h$ and $X^hu_h$.
Set $V_h=(M(|u_h|^p))^{1/p}$, and fix $1<t<p$. By
\eqref{eq_rellichneovjnrvojngrvjrn} and the uniform bounds on $u_h$
and $X^hu_h$,
\begin{equation}\label{bounddareverse111111}
      |\{V_h>\lambda\}|+|\{|X^hV_h|>\lambda\}|
    \leq C\lambda^{-p}
    \qquad\text{for every }\lambda>0
\end{equation}
for some $C\ges 0$ independent of $h$. Hence, for every measurable $E\subseteq\R^n$ with $0<|E|<\infty$,
the \emph{layer-cake formula} \cite[Remark 13.6]{MR2976521} and \eqref{bounddareverse111111} grant that
\begin{equation}\label{usodilayercakeformualecjnerijntv}
\begin{split}
      \int_E \bigl(V_h^t+|X^hV_h|^t\bigr)\,{\rm d}x
    &=t\int_0^\infty \lambda^{t-1}
      \bigl(|E\cap\{V_h>\lambda\}|
           +|E\cap\{|X^hV_h|>\lambda\}|\bigr)\,{\rm d}\lambda\\
           &\leq 2t|E|\int_0^{|E|^{-1/p}}\lambda^{t-1}\,{\rm d}\lambda
       +Ct\int_{|E|^{-1/p}}^\infty\lambda^{t-p-1}\,{\rm d}\lambda\\
       &=\left(2+\frac{Ct}{p-t}\right)|E|^{1-t/p}.
\end{split}
\end{equation}
By \eqref{usodilayercakeformualecjnerijntv} with $E=U$, we can apply \Cref{convimmin348598686}
with $s=t$: up to a subsequence, there exists $V\in L^t_{\rm loc}(U)$ such that 
$V_h\to V$ a.e. in $U$.
Since $\operatorname{supp}\chi\Subset U'$, then
$d=\operatorname{dist}(\operatorname{supp}\chi,\R^n\setminus U')>0$.
If $x\notin U'$, then $A_r(|u_h|^p)(x)=0$ for $r<d$. Instead, for $r\geq d$,  $A_r(|u_h|^p)(x)\leq |B(0,d)|^{-1}\|u_h\|_{L^p(\rr^n)}^p$.
Therefore,
\begin{equation}\label{uniformboundfuori}
        \sup_h\sup_{x\notin U'}V_h(x)^p
    \leq |B(0,d)|^{-1}\sup_h\|u_h\|_{L^p(\rr^n)}^p<\infty.
\end{equation}
We show that $L^p|\{x\in U':V(x)\geq2L\}|\to0$ as $L\to\infty$.
For $L^p$ larger than the uniform bound in \eqref{uniformboundfuori}, set
$T_L(s)=L^{p/t}\min\{(s/L-1)_+,1\}$. Then $T_L(V_h)$ is supported
in $\overline{U'}$ and, by \eqref{eq_rellichneovjnrvojngrvjrn}, \eqref{bounddareverse111111} and \eqref{usodilayercakeformualecjnerijntv},
\begin{equation}\label{stimaaccaelllerjefnritvjn}
     \|T_L(V_h)\|_t^t\leq L^p|\{V_h>L\}|\leq C,\qquad
 \|X^hT_L(V_h)\|_t^t
 \leq L^{p-t}\int_{\{V_h>L\}}|X^hV_h|^t\,{\rm d}x\leq C.
\end{equation}
For fixed $L$, since $V_h\to V$ a.e. in $U$ and $T_L$ is bounded, the dominated convergence theorem yields $T_L(V_h)\to T_L(V)$ in
$L^t(U)$. By \eqref{stimaaccaelllerjefnritvjn} and reflexivity, up to a
subsequence, $X^hT_L(V_h)$ converges weakly in $L^t(U;\R^m)$.
By \Cref{limitinsobolev},
$T_L(V)\in W_X^{1,t}(U)$, and $X^hT_L(V_h) \to XT_L(V)$ weakly.
Therefore, by \eqref{stimaaccaelllerjefnritvjn} and weak lower-semicontinuity, $(T_L(V))_L$ is bounded in $W^{1,t}_X(U)$. Moreover, by construction, 
$T_L(V)\in W^{1,t}_{X,0}(U)$. Notice that
$T_L(V)\to0$ a.e. as $L\to\infty$. Hence, \Cref{relkon} grants that $T_L(V)\to0$ strongly in $L^t(U)$. Therefore, since $T_L(V)^t=L^p$ on $\{V\geq2L\}$, we conclude that $L^p|\{x\in U':V(x)\geq2L\}|\to0$.
By \eqref{uniformboundfuori}, for $L$ large enough,
$\{V_h\geq L\}\subseteq U'$ for every $h$, and
\eqref{coda3848586885} gives
\begin{equation}\label{stimasenzalimitiejvnrtgvjntrgvo}
    \int_{\{|u_h|>L\}}|u_h|^p\,{\rm d}x
    \leq C_nL^p
    |\{x\in U':V_h(x)\ges L\}|.
\end{equation}
Since $V_h\to V$ a.e. in $U'$, at
a.e. point where $V(x)<L$ we have $V_h(x)<L$ for all
sufficiently large $h$. Thus
\begin{equation}\label{iduelimsup}
    \limsup_{h\to\infty}\mathbf 1_{\{V_h>L\}}
    \leq \mathbf 1_{\{V\geq L\}}
    \qquad\text{a.e. in }U'.
\end{equation}
Since $|U'|<\infty$, \eqref{stimasenzalimitiejvnrtgvjntrgvo}, \eqref{iduelimsup} and Fatou's inequality yield
\begin{equation*}
    \limsup_{h\to\infty}\int_{\{|u_h|>L\}} |u_h|^p\,{\rm d}x
    \leq C_n L^p |\{x\in U':V(x)\geq L\}|.
\end{equation*}
Since $L^p|\{x\in U':V(x)\geq2L\}|\to0$, we conclude that $(|u_h|^p)_h$ is uniformly integrable. Since $u_h\to u$ a.e., the thesis follows by Vitali's theorem (cf. \cite[Exercise 4.14]{MR2759829}).
\end{proof}

\begin{proof}[Proof of \Cref{nuovothmcompattezza}]
It suffices to consider the case $\varphi_h=\varphi=0$.
%
Assume first that $(X^h)_h\in\classunoraff$. By \Cref{primacasofacile},
$(u_h)_h$ is bounded in $W^{1,p}_{X,0}(\Om)$. The conclusion then follows from \Cref{relkon}.
Suppose that $(X^h)_h\in\classdueraff$ and $\Om$ has Lipschitz
boundary. Choose a bounded open set $U$ with
$\overline\Om\subset U\Subset\Om_0$, where $\Om_0$ is as in
\eqref{hormanderlimit}. Extend each  $\C^h$ so that \eqref{a0} and \eqref{ublip} hold
on $U$. Extend $u_h$ to be zero outside $\Om$. Since $u_h\in W^{1,p}_{X^h,0}(\Om)$, then $u_h\in W^{1,p}_{X^h}(U)$, and
its $X^h$-gradient on $U$ is the zero extension of $X^hu_h$.
Since $\overline\Om\subseteq U$,  \Cref{rellilocale3456789} implies that, up to a subsequence, $u_h\to u$ strongly in $L^p(\Om)$. Moreover, up to a further subsequence and 
by \Cref{limitinsobolev}, $u\in W^{1,p}_X(\Om)$ and $X^h u_h\to Xu$ weakly in $L^p(\Om;\rr^m)$. 
Since $\Om$ has Lipschitz boundary, we conclude by \Cref{stranachiusura} that
$u\in W^{1,p}_{X,0}(\Om)$.
\end{proof}
\color{black}
As a consequence of \Cref{nuovothmcompattezza}, we obtain a uniform Poincaré inequality.
\begin{proposition}[Uniform Poincaré inequality]\label[proposition]{convrayleight2}
Assume that either $(X^h)_h\in\classunoraff$  or
$(X^h)_h\in\classdueraff$ and $\Om$ has Lipschitz boundary.
Let $1<p<\infty$. Then
\begin{equation}\label{summadue}
    \lim_{h\to\infty}\mathcal R_h=\mathcal R.
\end{equation}
Moreover, $\mathcal R>0$, and there exists $h_0\in\N$ such that
\begin{equation}\label{upinstatement}
    \int_\Om |u|^p\,{\rm d}x
    \leq \frac{2}{\mathcal R}\int_\Om |X^h u|^p\,{\rm d}x
\end{equation}
for any $h\geq h_0$ and any $u\in W^{1,p}_{X^h,0}(\Om)$.
Finally, if $\varphi_h\in W^{1,p}_{X^h}(\Om)$, then
\begin{equation}\label{Poincarenonhomogeneousuniform}
    \int_\Om |u|^p\,{\rm d}x
    \leq \frac{2^{2p-1}}{\mathcal R}\int_\Om |X^h u|^p\,{\rm d}x
    +\frac{2^{2p-1}}{\mathcal R}\int_\Om |X^h\varphi_h|^p\,{\rm d}x
    +2^{p-1}\int_\Om |\varphi_h|^p\,{\rm d}x
\end{equation}
for any $h\geq h_0$ and any $u\in W^{1,p}_{X^h,\varphi_h}(\Om)$.
\end{proposition}

\begin{proof}
If $(X^h)_h\in\classunoraff$, \eqref{summadue}
follows from \Cref{convrayleighs1}. Assume that
$(X^h)_h\in\classdueraff$ and that $\Om$ has Lipschitz boundary.
Arguing \emph{verbatim} as in the proof of \Cref{convrayleighs1},
\begin{equation}\label{limsuprfac}
    \limsup_{h\to\infty}\mathcal R_h\leq\mathcal R.
\end{equation}
We prove the converse inequality. For any $h\in\N$ there exists
$u_h\in\mathbf{C}^\infty_c(\Om)$ such that $u_h\neq0$ and
\begin{equation}\label{raylimhfac}
    \frac{\int_\Om |X^hu_h|^p\,{\rm d}x}
         {\int_\Om |u_h|^p\,{\rm d}x}
    \leq\mathcal R_h+\frac1h.
\end{equation}
For any $h\in\N$, define $v_h=\|u_h\|_{L^p(\Om)}^{-1}u_h$.
Notice that, by \eqref{raylimhfac},
\begin{equation}\label{raylimh2fac}
    \frac{\int_\Om |X^hv_h|^p\,{\rm d}x}
         {\int_\Om |v_h|^p\,{\rm d}x}
    =\frac{\int_\Om |X^hu_h|^p\,{\rm d}x}
          {\int_\Om |u_h|^p\,{\rm d}x}
    \leq\mathcal R_h+\frac1h.
\end{equation}
Moreover, $\|v_h\|_{L^p(\Om)}=1$ for any $h\in\N$ and, by
\eqref{limsuprfac},
\[
    \limsup_{h\to\infty}\|X^hv_h\|_{L^p(\Om;\R^m)}^p
    \leq\limsup_{h\to\infty}\left(\mathcal R_h+\frac1h\right)
    \leq\mathcal R.
\]
Choose $(h_j)_j$ such that
$\mathcal R_{h_j}\to\liminf_{h\to\infty}\mathcal R_h$.
Applying \Cref{nuovothmcompattezza} to $(v_{h_j})_j$,
up to a further subsequence, there exists
$v\in W^{1,p}_{X,0}(\Om)$ such that $v_h\to v$ strongly in $L^p(\Om)$
and $X^hv_h\to Xv$ weakly in $L^p(\Om;\R^m)$. Moreover,
$\|v\|_{L^p(\Om)}=1$. Hence
\[
    \mathcal R
    \leq\int_\Om |Xv|^p\,{\rm d}x
    \leq\liminf_{h\to\infty}\int_\Om |X^hv_h|^p\,{\rm d}x
    \leq\liminf_{h\to\infty}\mathcal R_h,
\]
whence \eqref{summadue} holds.
Moreover, \Cref{poincaprelprop} gives $\mathcal R>0$. Thus \eqref{upinstatement} follows from the definition of
$\mathcal R_h$. Finally, applying \eqref{upinstatement} to $u-\varphi_h$
gives \eqref{Poincarenonhomogeneousuniform}.
\end{proof}
\section{\texorpdfstring{$\Gamma$}{Gamma}-convergence for functionals depending on moving anisotropies}\label{Section6}
Throughout the rest of the paper, $\Om\subseteq\R^n$ is a fixed bounded open set with Lipschitz boundary. In this section we establish several $\Gamma$-convergence properties for sequences of integral functionals depending on a sequence $(X^h)_h$ belonging to either $\classuno$ or $\classdue$.
Let us recall some basic preliminaries. For a complete account to $\Gamma$-convergence, as well as for the vocabulary of local functionals, we refer the reader to \cite{MR1968440,MR1201152}.
We recall that if $(\mathcal X,\tau)$ is a first-countable topological space, a sequence of functionals $(F_h)_h:\mathcal X\longrightarrow [0,\infty]$ is said to \emph{$\Gamma(\tau)$-converge} to a functional $F:\mathcal X\longrightarrow [0,\infty]$ if the following two conditions hold.
\begin{itemize}
    \item For any $u\in\mathcal X$ and any sequence $(u_h)_h$ converging to $u$ in $\tau$, then
    \begin{equation}\label{liminfineq}\tag{liminf}
        F(u)\leq\liminf_{h\to\infty}F_h(u_h).
    \end{equation}
    \item For any $u \in \mathcal X$, there exists a sequence $(u_h)_h$ converging to $u$ in $\tau$ such that
    \begin{equation}\label{recodef}\tag{limsup}
        F(u)\geqslant \limsup_{h\to\infty}F_h(u_h).
    \end{equation}
\end{itemize}
Sequences for which \eqref{recodef} holds are known as \emph{recovery sequences}. 
Equivalently, defining the \emph{$\Gamma$-lower limit} and \emph{$\Gamma$-upper limit} respectively by
 \begin{equation*}
    F'(u)\coloneqq \Gamma-\liminf_{h\to\infty}F_h(u)\coloneqq\inf\left\{\liminf_{h\to\infty}F_h(u_h)\,:\,u_h \to u \text{ in }\tau \right\}
 \end{equation*}
 and
 \begin{equation*}
     F''(u)\coloneqq\Gamma-\limsup_{h\to\infty}F_h(u)\coloneqq\inf\left\{\limsup_{h\to\infty}F_h(u_h)\,:\,u_h\to u\text{ in }\tau\right\},
 \end{equation*}
  $(F_h)_h$ $\Gamma$-converges to $F:(\mathcal X,\tau)\scu[0,\infty]$ if and only if
 \begin{equation*}
     \Gamma-\liminf_{h\to\infty}F_h(u)=\Gamma-\limsup_{h\to\infty}F_h(u)=F(u)
 \end{equation*}
 for any $u\in \mathcal X$. We say that $F$ is the \emph{$\Gamma-$limit of $(F_h)_h$} and we write $F=\Gamma-\lim_{h\to\infty}F_h$.
Throughout this section, we fix a sequence $(X^h)_h\in\classuno\cup\classdue$ and $1<p<\infty$. The wider class of functionals we are interested in is made of (local) integral functionals of the form 
\begin{align*}
    F_h(u,A)\coloneqq\displaystyle{\begin{cases}
    \int_{A} f_h(x,X^hu(x)){\rm d}x&\text{ if }A\in\mA,\,u\in L^p(\Om)\cap W^{1,p}_{X^h}(A),\\
    \infty&\text{ otherwise in $L^p(\Om)$},
\end{cases}}
\end{align*}
where each integrand $f_h$ belongs to the following class.
\begin{definition}
Fix $0<d_1\leq\,d_2$ and a non-negative function $a\in L^1(\Omega)$. Fix $h\in \N$. $\mathcal{I}^h_p(a,d_1,d_2)$ is the class of functions $f_h:\,\Omega\times\R^m\to\rr$ such that: 
\begin{itemize}
\item[$(I_1)$] $f_h:\,\Om\times\R^m\to\rr$ is a Carathéodory function;
\item[$(I_2)$] for a.e. $x\in\Om$, the function
$\xi\mapsto f_h(x,\C^h(x)\xi)$ is convex on $\R^n$;
\item[$(I_3)$] it holds that
\begin{equation}\label{3.2}
    d_1\,|\C^h(x)\xi|^p\leq f_h(x,\C^h(x)\xi)\leq a(x)+d_2\left|\C^h(x)\xi\right|^p
\end{equation}
for a.e. $x\in\Om$ and for every $\xi\in\R^n$.
\end{itemize}
\begin{remark}
    If $F_h$ is the local functional associated with some $f_h\in\mathcal{I}^h_p(a,d_1,d_2)$, then
    \begin{equation}\label{crescitahhhh}
    d_1\|X^hu\|^p_{L^p(A)}\leq F_h(u,A)\leq\int_A a(x)\,{\rm d}x+d_2\|X^hu\|^p_{L^p(A)}
\end{equation}
for any $A\in\mA$ and any $u\in W^{1,p}_{X^h}(A)$, by \cite[Theorem 4.3]{withoutlic}. Moreover, the same theorem and $(I_2)$
imply the convexity of $F_h(\cdot,A)$.
\end{remark}
\end{definition}
\subsection{\texorpdfstring{$\Gamma$}{Gamma}-convergence of norms}
Our first result focuses on the special case of norms. \Cref{normgammaconv} will be propaedeutic to the proof of the main compactness theorem, \Cref{gammathmwithproof}.
\begin{theorem}\label{normgammaconv}
Let $\Om$ be a bounded open set with Lipschitz boundary. Let $(X^h)_h\in\classuno\cup\classdue$. Let $1<p<\infty$.
For $h\in\N$, set $\Psi_p^h:L^p(\Om)\times\mA\longrightarrow[0,\infty]$ and $\Psi_p:L^p(\Om)\times\mA\longrightarrow[0,\infty]$ respectively by
\begin{equation*}
    \Psi_p^h(u,A)=\displaystyle{\begin{cases}
    \int_{A} |X^hu|^p\,d x&\text{ if }A\in\mA,\,u\in W^{1,p}_{X^h}(A),\\
    \infty&\text{otherwise},
    \end{cases}}
\end{equation*}
and
\begin{equation*}
    \Psi_p(u,A)=\displaystyle{\begin{cases}
    \int_{A} |Xu|^p\,d x&\text{ if }A\in\mA,\,u\in W^{1,p}_{X}(A),\\
    \infty&\text{otherwise}.
    \end{cases}}
\end{equation*}
Then,
\begin{equation*}
    \Psi_p(\cdot,A)=\Gamma(L^p)-\lim_{h\to\infty}\Psi_p^h(\cdot,A)
\end{equation*}
for any $A\in\mA$ with Lipschitz boundary.
\end{theorem}
\begin{proof}
Fix $A\in\mA$ with Lipschitz boundary and $u\in L^p(\Om)$.
Assume first that $u\notin W^{1,p}_X(A)$.
Then, $\Psi_p(u,A)=\infty,$
so that \eqref{recodef} follows with $u_h=u$.
On the other hand, if there exists $(u_h)_h\subseteq L^p(\Om)$ such that $u_h\to u$ in $L^p(\Om)$ and $\liminf_{h\to\infty}\Psi_p^h(u_h,A)<\infty$, then either \Cref{primacasofacile} or \cref{limitinsobolev} would imply that $u\in W^{1,p}_{X}(A)$, a contradiction. 
Instead, assume $u\in W^{1,p}_X(A)$. 
The existence of a sequence as in \eqref{recodef} follows by applying \cref{meyser} on $A$ and setting $u_h=u$ on $\Om\setminus A$, while \eqref{liminfineq} follows as above in view of either \Cref{primacasofacile} or \cref{limitinsobolev}.
\end{proof}
\begin{remark}
    By \Cref{esempiononvaleleimonfifr},  \Cref{normgammaconv} is in general false under the sole assumption \eqref{a0}.
\end{remark}
\begin{remark}
    As in \cref{normgammaconv}, one may show $\Gamma$-convergence of complete Sobolev norms.
\end{remark}
\subsection{\texorpdfstring{$\Gamma$}{Gamma}-compactness in the general case}
By \Cref{normgammaconv}, we prove compactness in full generality. 
\begin{theorem}\label{gammathmwithproof}
Let $\Om$ be a bounded open set with Lipschitz boundary. Let $(X^h)_h\in\classuno\cup\classdue$. Let $1<p<\infty$. Let $a\in L^1(\Om)$, $a\geq 0$, and $0< d_1\leq d_2$.
Let $(f_h)_h\subseteq\mathcal{I}^h_p(a,d_1,d_2)$. Denote by $(F_h)_h:L^p(\Om)\times\mA\longrightarrow [0,\infty]$ the sequence of integral functionals such that
\begin{equation}\label{708funzionaleduedef24}
    F_h(u,A)=\displaystyle{\begin{cases}
    \int_{A} f_h(x,X^hu(x))\,d x&\text{if }A\in\mA,\,u\in W^{1,p}_{X^h}(A),\\
    \infty&\text{otherwise}.
    \end{cases}}
\end{equation}
Then there exist $F:L^p(\Om)\times\mA\longrightarrow [0,\infty]$ and $f\in\mathcal{I}^\infty_p(a,d_1,d_2)$ such that, up to a subsequence,
\begin{equation}\label{gammaconvstatement}
    F(\cdot,A)=\Gamma(L^p)-\lim_{h\to\infty}F_{h}(\cdot,A)
\end{equation}
for any $A\in\mA$ with Lipschitz boundary, and $F$ is represented by
\begin{equation}\label{708funzionaleduedef243}
    F(u,A)=\displaystyle{
    \begin{cases}
    \int_{A} f(x,Xu(x))\,d x&\text{if }A\in\mA,\,u\in W^{1,p}_{X}(A),\\
    \infty&\text{otherwise}.
\end{cases}}
\end{equation}
Moreover, $f$ can be uniquely chosen in $\mathcal{I}^\infty_p(a,d_1,d_2)$ in such a way that \begin{equation}\label{costonkeralbefabio}
    f(x,\eta)=f(x,\C(x)\xi_\eta)
\end{equation}
for a.e. $x\in\Om$ and any $\eta\in\R^m$, where $\xi_\eta$ is any vector in $\R^n$ such that 
\begin{equation}\label{ortodecomp}
    \eta-\C(x)\xi_\eta\in\operatorname{Im}\C(x)^\perp.
\end{equation}
\end{theorem}
\begin{remark}\label{spiegoneuno}
   Although reminiscent of the approach presented in \cite{MR1201152}, ours differs from the former.
Indeed, the classical approach to $\Gamma$-compactness consists of three main steps.
\begin{enumerate}
    \item Given a sequence of integral functionals $(F_h)_h$, extract a subsequence, say again $(F_h)_h$ which $\bar\Gamma$-converges to a local functional $F$, where $\bar \Gamma$-convergence (cf. \cite[Chapter 16]{MR1201152}) is a variational convergence tailored to local functionals.
    \item Upgrade $\bar\Gamma$-convergence to $\Gamma$-convergence.
    \item Show that $F$ is an integral functional.
\end{enumerate}
The proof of (2) (cf. e.g. \cite[Proposition 18.6]{MR1201152}) relies on uniform growth conditions like
\begin{equation*}
    F_h(u,A)\leq \tilde G(u,A)
\end{equation*}
for a lower-semicontinuous measure $\tilde G$.
Nevertheless, we only dispose of non-uniform growth conditions 
\begin{equation*}
    F_h(u,A)\leq \tilde G_h(u,A)\coloneqq\int_A a(x)\,{\rm d}x+d_2\|X^hu\|^p_{L^p(A)}.
\end{equation*}
Nevertheless, \Cref{normgammaconv} ensures that $(\tilde G_h)_h$ $\Gamma$-converges. We will show that this weaker framework will be still sufficient to promote $\bar\Gamma$-convergence to $\Gamma$-convergence. In addition, once $\Gamma$-convergence is achieved, in order to conclude we will need to apply the anisotropic integral representation result \cite[Theorem 1.1]{withoutlic}.
\end{remark}
\begin{proof}[Proof of \Cref{gammathmwithproof}]
\textbf{Step 1.} According to \cite[Theorem 19.4]{MR1201152}, $F:\lp{\Om}\times\mA\scu[0,\infty]$ belongs to $\clam$ if $F$ is a measure and if there exist $e_1\geq 1$, $e_2,e_3,e_4\geq 0$, a finite measure $\mu$, independent of $F$, and a measure $G:\lp{\Om}\times\mA\scu[0,\infty]$, which may depend on $F$, such that
\begin{equation}\label{708ufe1}
    G(u,A)\leq F(u,A)\leq e_1 G(u,A)+e_2\|u\|^p_{\lp{A}}+\mu(A)
\end{equation}
and
\begin{equation}\label{708ufe2}
    G(\varphi u+(1-\varphi)v,A)\leq e_4(G(u,A)+G(v,A))+e_3e_4(\max|D\varphi|^p)\|u-v\|^p_{L^p(A)}+\mu(A),
\end{equation}
for any $u,v\in\lp{\Om}$, 
 any $A\in\mA$ and any $\varphi\in \mathbf{C}^\infty _c(\Om)$ such that $0\leq\varphi\leq 1$.
We claim that $(F_h)_h\subseteq\clam$. Fix $h\in\N$, and set $\mu(B)\coloneqq\int_Ba(x) \, {\rm d} x.$
Then $\mu$ is a finite measure on $\Om$. Moreover, in view of \cite[Lemma 4.14]{MR4108409}, the non-negative local functional $G_h\coloneqq d_1\Psi^h_p(u,A)$
is a measure. Hence, letting $e_1\coloneqq\frac{d_2}{d_1}$ and $e_2\coloneqq0$, \eqref{708ufe1} follows from \eqref{crescitahhhh}. 
Moreover, recalling \eqref{a0} and arguing \emph{verbatim} as in the proof of \cite[Proposition 3.3]{MR4566142}, \eqref{708ufe2} follows for suitable constants $e_3,e_4\geq 0$ independent of $h\in\N$.
\\
\textbf{Step 2.}
By \cite[Theorem 19.5]{MR1201152}, 
there exists an $L^p$-lower semicontinuous functional $F\in\clam$ such that, up to a subsequence, $(F_h)_h$ $\overline{\Gamma}(L^p)$-converges to $F$. By definition of $\clam$, $F$ is a measure. Moreover, by \cite[Proposition 16.15]{MR1201152}, $F$ is local.\\
\textbf{Step 3.}
We claim that
\begin{equation}\label{crescita}
    d_1\|Xu\|^p_{L^p(A)}\leq F(u,A)\leq\int_A a(x)\,{\rm d}x+d_2\|Xu\|^p_{L^p(A)}\qquad\text{for any $A\in\mA$, any $u\in W^{1,p}_X(A)$}.
\end{equation}
 Indeed, setting $G\coloneqq d_1\Psi_p$, we know from \Cref{normgammaconv} that
\begin{equation}\label{boundsmeasure}
    G(\cdot,A)=\Gamma(L^p)-\lim_{h\to\infty}G_h(\cdot,A)
\end{equation}
for any $A\in\mA$ with Lipschitz boundary. Every open set admits an increasing exhaustion by smooth open sets compactly contained in it. Since $G(u,\cdot)$ is inner regular, the inner regular envelopes of the $\Gamma$-lower and upper limits of $(G_h)_h$ both equal $G$ on $\mA$ (cf. \cite[Proposition 16.4]{MR1201152}). Thus $(G_h)_h$ $\overline{\Gamma}(L^p)$-converges to $G$. Therefore, passing to the $\overline{\Gamma}(L^p)$-limit in $ G_h(u,A)\leq F_h(u,A)\leq e_1 G_h(u,A)+\mu(A)$
gives $G(u,A)\leq F(u,A)\leq e_1G(u,A)+\mu(A)$ on every $A\in\mA$, and in particular \eqref{crescita}.\\
\textbf{Step 4.}
We claim that \begin{equation}\label{gammainproof}
    F(\cdot,A)=\Gamma(L^p)-\lim_{h\to\infty}F_h(\cdot,A)
\end{equation}
for any $A\in\mA$ with Lipschitz boundary. Fix such an $A$ and $u\in W^{1,p}_X(A)\cap L^p(\Om)$. 
We claim that
\begin{equation}\label{quasigammaconv}
    F(u,A)=F'(u,A)=F''(u,A).
\end{equation}
By \cite[Proposition 16.4]{MR1201152}, we already know that $F(u,A)\leq F'(u,A)\leq F''(u,A)$, so that we are left to show that $F''(u,A)\leq F(u,A)$. Following the proof of \cite[Proposition 18.6]{MR1201152}, and noticing that $\nu(B)\coloneqq\int_Ba\,{\rm d}x+d_2\Psi_p(u,B)$ is a finite measure on $A$, for any $\varepsilon>0$ we choose $K_\varepsilon\Subset A$, the closure of a smooth open set, such that $\nu(A\setminus K_\varepsilon)\leq\varepsilon$. Moreover, recalling \eqref{boundsmeasure} and by \cite[Proposition 6.7]{MR1201152}, we infer that
\begin{equation}\label{stimaconmisura}
    F''(u,\tilde A)\leq \nu(\tilde A)
\end{equation}
for any $\tilde A\in\mA$ with Lipschitz boundary, whence in particular for $\tilde A=A\setminus K_\varepsilon$. Choose two open sets $A',A''\in\mA$ such that $K_\varepsilon\Subset A'\Subset A''\Subset A.$
Since $(F_h)_h\subseteq\clam$, \cite[Theorem 19.4]{MR1201152} implies that $(F_h)_h$ satisfies the \emph{uniform fundamental estimate} (cf. \cite[Definition 18.2]{MR1201152}). Therefore, noticing that $  A=A'\cup(A\setminus K_\varepsilon),$
we can apply \cite[Proposition 18.3]{MR1201152} to infer that
\begin{align*}
    F''(u,A)
    \leq F''(u,A'')+F''(u,A\setminus K_\varepsilon)
    \leq F(u,A)+\nu(A\setminus K_\varepsilon)
    \leq F(u,A)+\varepsilon,
\end{align*}
where in the semi-last inequality we exploited \cite[Remark 16.1]{MR1201152} together with \eqref{stimaconmisura}. Letting $\varepsilon\to 0$, \eqref{quasigammaconv} follows. 
If $u\notin W^{1,p}_X(A)$, Step 3 gives $F(u,A)\geq G(u,A)=\infty$. Since $F(u,A)\leq F'(u,A)\leq F''(u,A)$, \eqref{quasigammaconv} follows in this case as well. Therefore, \eqref{gammainproof} holds.\\
\textbf{Step 5.}
By Step 2, $F$ is $L^p$-lower semicontinuous. Since $F_h(u+c,A)=F_h(u,A)$ for every $c\in\R$, the same identity holds for the $\Gamma$-lower limits, and hence for their inner regular envelopes. Thus $F(u+c,A)=F(u,A)$ for every $A\in\mA$.
In conclusion, $F$ satisfies the hypotheses of \cite[Theorem 1.1]{withoutlic}. The thesis follows.
\end{proof}
\begin{remark}
    The very same approach to \Cref{gammathmwithproof}, together with the integral representation results provided in \cite{withoutlic}, allows to establish $\Gamma$-compactness results for integral functionals associated with complete Lagrangians $f(x,u,\eta)$ under suitable natural assumptions (cf. \cite{withoutlic}).
\end{remark}
\subsection{\texorpdfstring{$\Gamma$}{Gamma}-convergence for functionals with Dirichlet boundary conditions}

\color{black}
\Cref{gammathmwithproof} can be exploited to take into account Dirichlet boundary conditions.
\begin{theorem}\label{mainthmdirichletbc}
Let $\Om$ be a bounded open set with Lipschitz boundary. Let $(X^h)_h\in \classuno\cup\classdue$. 
Let $1<p<\infty$. Let $(f_h)_h\subseteq\mathcal{I}^h_p(a,d_1,d_2)$ and $f\in \mathcal{I}^\infty_p(a,d_1,d_2)$, and let $F_h,F:L^p(\Om)\times\mA\to[0,\infty]$ be the corresponding integral functionals, respectively defined in \eqref{708funzionaleduedef24} and \eqref{708funzionaleduedef243}.
Fix $(\varphi_h)_h\subseteq W^{1,p}_{X^h}(\Om)$ and $\varphi\in W^{1,p}_{X}(\Om)$ such that 
\begin{equation}\label{quasichiusuradebolehpduetre}
        \varphi_h\to\varphi\text{ strongly in }L^p(\Om)\qquad\text{and}\qquad X^h\varphi_h\to X\varphi\text{ strongly in }L^p(\Om;\R^m).
    \end{equation}
Define the functionals $F^{\varphi_h}_h,F^\varphi:L^p(\Om)\to[0,\infty]$ by
\begin{equation}\label{affinelocalfunctionalsquadratic}
    F^{\varphi_h}_h(u)=\begin{cases}
    F_h(u,\Om)\quad&\text{if }u\in W^{1,p}_{X^h,\varphi_h}(\Om),\\
    \infty\quad&\text{otherwise},
    \end{cases}\qquad F^{\varphi}(u)=\begin{cases}
    F(u,\Om)\quad&\text{if }u\in W^{1,p}_{X,\varphi}(\Om),\\
    \infty\quad&\text{otherwise}.
    \end{cases}
\end{equation}
Assume that 
\begin{equation}\label{gammalimsupdirichletno}
     F(u,\Om)= \Gamma(L^p)-\lim_{h\to\infty}F_{h}(u,\Om).
\end{equation}
Then
\begin{equation}\label{gammalimsupdirichlet}
     F^\varphi(u)= \Gamma(L^p)-\lim_{h\to\infty}F^{\varphi_h}_{h}(u).
\end{equation}
\end{theorem}
\begin{proof}
  First we prove that  $F^\varphi(u)\leq  \Gamma(L^p)-\liminf_{h\to\infty}F^{\varphi_h}_{h}(u)$. Let $u\in L^p(\Om)$, and let $(u_h)_h\subseteq L^p(\Om)$ be such that $u_h\to u$ strongly in $L^p(\Om)$. If $\liminf_{h\to\infty}F^{\varphi_h}_h(u_h)=\infty$, \eqref{liminfineq} follows. 
Instead, up to a subsequence and by \eqref{crescitahhhh} and \eqref{quasichiusuradebolehpduetre}, $(X^h(u_h-\varphi_h))_h$ is bounded in $L^p(\Om;\R^m)$. In $\classuno$, \Cref{primacasofacile} implies that $(u_h-\varphi_h)_h$ is bounded in $W^{1,p}_{X,0}(\Om)$, therefore, up to a subsequence, $u_h-\varphi_h\to u-\varphi$ weakly in $W^{1,p}_{X}(\Om)$. In $\classdue$, by \Cref{limitinsobolev} and \eqref{quasichiusuradebolehpduetre}, $X^h u_h\to Xu$ weakly in $L^p(\Om;\rr^m)$, whence $u\in W^{1,p}_{X,\varphi} (\Om)$ by \Cref{stranachiusura}. In both cases $u\in W^{1,p}_{X,\varphi}(\Om)$. Consequently,
    \begin{equation*}
        F^\varphi(u)=F(u,\Om)\overset{\eqref{gammalimsupdirichletno}}{\leq}\liminf_{h\to\infty} F_h(u_h,\Om)=\liminf_{h\to\infty}F_h^{\varphi_h}(u_h),
    \end{equation*}
    whence the first claim follows. Now we prove that  $F^\varphi(u)\geq  \Gamma(L^p)-\limsup_{h\to\infty}F^{\varphi_h}_{h}(u)$. Let $u\in L^p(\Om)$. Fix $\varepsilon>0$. Again, we may assume that $u\in W^{1,p}_{X,\varphi}(\Om)$. Hence, by \Cref{MeySerAffine} there exists a sequence $(v_h)_h\subseteq \mathbf{C}^\infty _c(\Om)$ such that $v_h\to u-\varphi$ strongly in $L^p(\Om)$ and $X^h v_h\to X(u-\varphi)$ strongly in $L^p(\Om;\R^m)$. By \eqref{quasichiusuradebolehpduetre}, $|X^h(v_h+\varphi_h)|^p\to |Xu|^p$ in $L^1(\Om)$. Hence \eqref{crescitahhhh} gives a compact set $K_\varepsilon\Subset\Om$ such that, for all sufficiently large $h$,
    \begin{equation}\label{stimaepsilon}
        F_h(v_h+\varphi_h,\Om\setminus K_\varepsilon)\leq\varepsilon.
    \end{equation}
    By \eqref{gammalimsupdirichletno} there exists $(u_h)_h\subseteq W^{1,p}_{X^h}(\Om)$ such that $u_h\to u$ strongly in $L^p(\Om)$ and $F(u,\Om)\geq\limsup_{h\to\infty}F_h(u_h,\Om)$.
    Let $A',A''\in\mA$ be such that $K_\varepsilon\subseteq A'\Subset A''\Subset\Om$. By the proof of \Cref{gammathmwithproof}, we know that $(F_h)_h$ satisfies the uniform fundamental estimate as in \cite[Definition 18.2]{MR1201152}. Therefore, there exists $M>0$ and a sequence of smooth cut-off functions between $A'$ and $A''$, say $(\psi_h)_h$, such that, for all sufficiently large $h$,
    \begin{equation}\label{stimadastimafond}
    \begin{split}
           F_h(\psi_h u_h&+(1-\psi_h)(v_h+\varphi_h),\Om)\leq (1+\varepsilon)\left(F_h(u_h,\Om)+F_h(v_h+\varphi_h,\Om\setminus K_{{\varepsilon}})\right)\\
           &\quad+\varepsilon\left(\|u_h\|_{L^p(\Om)}^p+\|v_h+\varphi_h\|^p_{L^p(\Om)}+1\right)+M\|u_h-(v_h+\varphi_h)\|^p_{L^p(\Om)}\\
           &\overset{\eqref{quasichiusuradebolehpduetre},\eqref{stimaepsilon}}{\leq} (1+\varepsilon)F_h(u_h,\Om)+\varepsilon(1+\varepsilon)+\varepsilon\left(2\|u\|_{L^p(\Om)}^p+1\right)+\varepsilon.
    \end{split}
    \end{equation}
    Set $w_h\coloneqq\psi_h u_h+(1-\psi_h)(v_h+\varphi_h)$ for any $h\in\N$. Notice that $w_h\to u$ strongly in $L^p(\Om)$. Fix $h\in N$. As $u_h-\varphi_h\in W^{1,p}_{X^h}(\Om)$, by \Cref{MeySer} there exists a sequence $(u^h_k)_k\subseteq \mathbf{C}^\infty (\Om)\cap W^{1,p}_{X^h}(\Om)$ such that $u^h_k\to u_h-\varphi_h$ strongly in $L^p(\Om)$ and $X^hu^h_k\to X^hu_h-X^h\varphi_h$ strongly in $L^p(\Om;\R^m)$. Notice that 
    \begin{equation*}
        w_h-\varphi_h=\psi_h(u_h-\varphi_h)+(1-\psi_h)v_h.
    \end{equation*}
    For any $k\in\N$, set $w^h_k=\psi_hu^h_k+(1-\psi_h)v_h$. As $v_h\in \mathbf{C}^\infty _c(\Om)$,  $(w^h_k)_k\subseteq \mathbf{C}^\infty _c(\Om) $. Since
    \begin{equation*}
       w^h_k-(w_h-\varphi_h)=\psi_h(u^h_k-(u_h-\varphi_h)),
    \end{equation*}
    we conclude that $w^h_k\to w_h-\varphi_h$ strongly in $L^p(\Om)$ and $X^hw^h_k\to X^h(w_h-\varphi_h)$ strongly in $L^p(\Om;\R^m)$. Therefore $w_h\in W^{1,p}_{X^h,\varphi_h}(\Om)$, and in particular $ F_h(w_h,\Om)=F_h^{\varphi_h}(w_h)$. Combining this information with \eqref{stimadastimafond}, we infer
    \begin{equation}\label{stimautilissimaepsilon}
        F^{\varphi_h}_h(w_h)\leq (1+\varepsilon)F_h(u_h,\Om)+\varepsilon(1+\varepsilon)+\varepsilon\left(2\|u\|_{L^p(\Om)}^p+1\right)+\varepsilon,
    \end{equation}
whence
\begin{equation*}
\begin{split}
      \Gamma(L^p)-\limsup_{h\to\infty}&F^{\varphi_h}_h(u)\leq\limsup_{h\to\infty}F^{\varphi_h}_h(w_h)\\
      &\overset{\eqref{stimautilissimaepsilon}}{\leq }(1+\varepsilon)\limsup_{h\to\infty}F_h(u_h,\Om)+\varepsilon(1+\varepsilon)+\varepsilon\left(2\|u\|_{L^p(\Om)}^p+1\right)+\varepsilon\\
      &=(1+\varepsilon)F(u,\Om)+\varepsilon(1+\varepsilon)+\varepsilon\left(2\|u\|_{L^p(\Om)}^p+1\right)+\varepsilon\\
       &=(1+\varepsilon)F^\varphi(u)+\varepsilon(1+\varepsilon)+\varepsilon\left(2\|u\|_{L^p(\Om)}^p+1\right)+\varepsilon.
\end{split}
\end{equation*}
Being $\varepsilon$ arbitrary, the thesis follows.
\end{proof}
\begin{remark}\label{spiegonedue}
    Although reminiscent of the proof of \cite[Theorem 21.1]{MR1201152}, our argument needs to differ from the former: as the finiteness domain of $F_h^{\varphi_h}$ depends on $h$, the limit function $u$ in the proof of \eqref{recodef} cannot be used to construct a recovery sequence, but has to be properly well-approximated. On the other hand, our approach allows to consider a sequence of evolving boundary data instead of a fixed one.
\end{remark}
\subsection{Convergence of minima and minimizers}
Once \Cref{nuovothmcompattezza}, \Cref{convrayleight2} and \Cref{mainthmdirichletbc} are established, we argue \emph{verbatim} as in the Euclidean setting to infer convergence of minima and minimizers. In view of \Cref{Section7}, we additionally take into account a continuous perturbation.
\begin{theorem}\label{minimaminimizerthm}
   Let $\Om$ be a bounded open set with Lipschitz boundary. Let $(X^h)_h\in \classunoraff\cup\classdueraff$. 
Let $1<p<\infty$. Let $(f_h)_h\subseteq\mathcal{I}^h_p(a,d_1,d_2)$ and $f\in \mathcal{I}^\infty_p(a,d_1,d_2)$.
Let $(\varphi_h)_h\subseteq W^{1,p}_{X^h}(\Om)$ and $\varphi\in W^{1,p}_{X}(\Om)$ satisfy \eqref{quasichiusuradebolehpduetre}.
Let $F^{\varphi_h}_h,F^\varphi:L^p(\Om)\to[0,\infty]$ be as in \eqref{affinelocalfunctionalsquadratic}.
Assume that \eqref{gammalimsupdirichlet} holds. 
Let $G:L^p(\Om)\to\R$ be  continuous and such that, for any $\varepsilon>0$, there are $\delta_1=\delta_1(\varepsilon)>0$, $\delta_2=\delta_2(\varepsilon)>0$ and $\delta_3=\delta_3(\varepsilon)>0$ such that
\begin{equation}\label{epsilondelta}
    -\delta_1(\varepsilon)-\varepsilon \int_\Om|u|^p\,{\rm d}x\leq G(u)\leq \delta_2(\varepsilon)+\delta_3(\varepsilon)\int_\Om |u|^p\,{\rm d}x
\end{equation}
for any $u\in L^p(\Om)$.
Up to a subsequence, the following properties hold.
\begin{enumerate}
\item  $F^\varphi+G$ has a minimum, and
  \begin{equation*}
        \inf_{L^p(\Om)}(F^{\varphi_h}_h+G)\to  \min_{L^p(\Om)}(F^{\varphi}+G). 
    \end{equation*}
    \item Assume that there exists a sequence $(u_h)_h\subseteq W^{1,p}_{X^h,\varphi_h}(\Om)$ such that, for any $h\in N$, $u_h$ minimizes $F^{\varphi_h}_h+G$. Then there exists a minimum point $u$ of $F^\varphi+G$ such that
    \begin{equation*}
         u_{h}\to u \text{ strongly in $L^p(\Om)$}. 
    \end{equation*}
\end{enumerate} 
\end{theorem}
\begin{proof}
     As $G$ is continuous, \eqref{gammalimsupdirichlet} and \cite[Proposition 6.21]{MR1201152} imply that 
    \begin{equation*}\label{gammaconvsum}
        F^\varphi+G=\Gamma(L^p)-\lim_{h\to\infty}\left(F^{\varphi_h}_h+G\right).
    \end{equation*}
    For any $h\in\N$, set $m_h=\inf_{L^p(\Om)}(F^{\varphi_h}_h+G)$. Notice that, for sufficiently large $h$, $m_h\in\R$. Let $u_h\in L^p(\Om)$ be such that $m_h\leq F^{\varphi_h}_h(u_h)+G(u_h)\leq m_h+\frac{1}{h}$. Then, for any $\varepsilon>0$, there exists $\hat M=\hat M(\varepsilon)>0$ such that
    \begin{equation*}
    \begin{split}
           F^{\varphi_h}_h(u_h)+G(u_h)&\leq m_h+1\\
           &\leq F^{\varphi_h}_h(\varphi_h)+G(\varphi_h)+1\\
           &\overset{\eqref{crescitahhhh},\eqref{epsilondelta}}{\leq } \int_\Om a\,{\rm d}x+\delta_2(\varepsilon)+\delta_3(\varepsilon)\int_\Om|\varphi_h|^p\,{\rm d}x+d_2\int_\Om|X^h\varphi_h|^p\,{\rm d}x+1\\
           &\overset{\eqref{quasichiusuradebolehpduetre}}{\leq }\hat M(\varepsilon).
    \end{split}
    \end{equation*}
    In particular, by \Cref{convrayleight2},
    \begin{equation*}
    \begin{split}
        \hat M(\varepsilon)&\geq F^{\varphi_h}_h(u_h)+G(u_h)\\
        &\overset{\eqref{crescitahhhh},\eqref{epsilondelta}}{\geq}d_1\int_\Om|X^hu_h|^p\,{\rm d}x-\varepsilon \int_\Om|u_h|^p\,{\rm d}x-\delta_1(\varepsilon)\\
        & \overset{\eqref{Poincarenonhomogeneousuniform}}{\geq}
        \left(d_1-\frac{\varepsilon 2^{2p-1}}{\mathcal R}\right)\int_\Om|X^h u_h|^p\,{\rm d}x-\frac{\varepsilon 2^{2p-1}}{\mathcal R}\int_\Om|X^h \varphi_h|^p\,{\rm d}x-\varepsilon 2^{p-1}\int_\Om|\varphi_h|^p\,{\rm d}x-\delta_1(\varepsilon)
    \end{split}    
    \end{equation*}
    Combining the previous inequalities with \Cref{convrayleight2}, $(u_h)_h$ is bounded in $L^p(\Om)$ and $(X^hu_h)_h$ is bounded in $L^p(\Om;\R^m)$. By \Cref{nuovothmcompattezza} and up to a subsequence, there exists $u\in W^{1,p}_{X,\varphi}(\Om)$ such that $u_h\to u$ strongly in $L^p(\Om)$. The rest of the thesis follows as in \cite[Theorem 7.2]{MR1201152}.
\end{proof}
\begin{remark}\label{acosaloapplichiamo}
    For any $\mu\geq 0$ and any $g\in L^{p'}(\Om)$, set $G:L^p(\Om)\to\R$ by
    \begin{equation*}
        G(u)=\frac{\mu}{p}\int_\Om |u|^p\,{\rm d}x-\int_\Om gu\,{\rm d}x
    \end{equation*}
    for any $u\in L^p(\Om)$. It is easy to check that $G$ is continuous and satisfies \eqref{epsilondelta}.
\end{remark}
\subsection{Quadratic forms}
In this section we let $p=2$, and we specialize \cref{gammathmwithproof} to the class of quadratic forms. Moreover, as this will be important in the forthcoming \Cref{Section7}, we provide a convergence result of the so-called \emph{momenta}, defining the latter to keep into account the moving anisotropic setting (cf. \Cref{hconvdef}).
\begin{definition}
Fix $h\in\N$ and $0<\lambda\leq\Lambda$. We denote by $E^h(\Omega;\lambda,\Lambda)$ the class of symmetric, measurable matrix-valued functions $A^h:\Omega\to\R^{m\times m}$ such that
\begin{equation}\label{growthforHconv}
    \lambda|\C^h(x)\xi|^2\leq\langle A^h(x)\C^h(x)\xi,\C^h(x)\xi\rangle_{\R^m}\leq \Lambda|\C^h(x)\xi|^2
\end{equation}
for a.e. $x\in\Om$ and for every $\xi\in\R^n$.
%
\end{definition}
Quadratic forms 
corresponding to 
elements of $E^h(\Omega;\lambda,\Lambda)$ express as
\begin{equation}\label{localfunctionalsquadratic}
    F_h(u,A)=\begin{cases}
    \int_A\langle A^h(x)X^hu(x),X^hu(x)\rangle_{\R^m}{\rm d}x\quad&\text{if }u\in W^{1,2}_{X^h}(A),\\
    \infty\quad&\text{otherwise}.
    \end{cases}
\end{equation}
\begin{theorem}\label{compactnessL2perHconvergenza}
Let $\Om$ be a bounded open set with Lipschitz boundary. Let $(X^h)_h\in\classuno\cup\classdue$. Let $0<\lambda\leq\Lambda$.
Let $(A^h)_h\subseteq E^h(\Omega;\lambda,\Lambda)$. For any $h\in\N$, denote by $F_h$ the corresponding quadratic form in the sense of \eqref{localfunctionalsquadratic}. There exists $A\in E^\infty(\Omega;\lambda,\Lambda)$ and a corresponding quadratic form $F:L^2(\Om)\times\mA\longrightarrow [0,\infty]$
%
such that, up to subsequences, \eqref{gammaconvstatement} holds for every $A\in\mA$ with Lipschitz boundary.
\end{theorem}
\begin{proof}
Let $f_h$ be the Lagrangian associated with $F_h$.
Then $(f_h)_h\subseteq\mathcal{I}^h_2(0,\lambda,\Lambda)$.
In virtue of \cref{gammathmwithproof}, there exist $f\in\mathcal{I}_2^\infty(0,\lambda,\Lambda)$ and a corresponding integral functional $F:L^2(\Om)\times\mA\longrightarrow [0,\infty]$
such that, up to subsequences, \eqref{gammaconvstatement} holds for every $A\in\mA$ with Lipschitz boundary. 
Fix $A\in\mA$ with Lipschitz boundary. By \eqref{gammaconvstatement} and \cite[Theorem 11.10]{MR1201152}, there is a symmetric bilinear form $B:W^{1,2}_{X}(A)\times W^{1,2}_{X}(A)\longrightarrow \R$ such that
\begin{equation*}
    \int_A f(x,Xu)\,{\rm d}x=B(u,u)
\end{equation*}
for any $u\in W^{1,2}_X(A)$. Arguing \emph{verbatim} as in the proof of \cite[Theorem 22.1]{MR1201152}, there exists an $n\times n$ symmetric, measurable matrix-valued function $A_e$ such that
\begin{align*}
  f(x,\C(x)\xi)=\langle A_e(x)\xi,\xi\rangle_{\R^n}
\end{align*}
for a.e. $x\in\Om$ and for every $\xi\in\R^n$. Arguing as in \cite{MR1417720,withoutlic}, there exists a unique, measurable, $n\times m$ matrix-valued function $\Cm(x)$ such that
    \begin{equation*} \label{pseudoinvproperties}
    \begin{aligned}
        &\Cm(x)\C(x)\Cm(x)=\Cm(x),\qquad \C(x)\Cm(x)\C(x)=\C(x),\\
        &\Cm(x)\C(x)=\C(x)^T\Cm(x)^T,\qquad \C(x)\Cm(x)=\Cm(x)^T\C(x)^T
    \end{aligned}
    \end{equation*} 
    for any $x\in\Om$. Set $A(x)=\Cm(x)^TA_e(x)\Cm(x)$. Fix $x\in\Om$ and $\eta\in\R^m$, and let $\xi_\eta$ be as in \eqref{ortodecomp}. Let
\begin{equation*}\label{N&Vx}
		N_x=\ker(\C(x))\qquad\text{and}\qquad V_x=\left\{\C(x)^T\eta\ :\ \eta\in\mathbb{R}^m\right\}.
	\end{equation*}
	From standard linear algebra,
    there are unique $(\xi_\eta)_{N_x}\in N_x$ and $(\xi_\eta)_{V_x}\in V_x$ such that
	\begin{equation*}\label{splitting}
		\xi_\eta=(\xi_\eta)_{N_x}+\,(\xi_\eta)_{V_x}.
	\end{equation*}
    By \cite[Lemma 3.13]{MR4108409},
    \begin{equation*}
         \langle A_e(x)\xi_\eta,\xi_\eta\rangle_{\R^n}= \langle A_e(x)(\xi_\eta)_{V_x},(\xi_\eta)_{V_x}\rangle_{\R^n}.
    \end{equation*}
    Moreover, by \cite[Proposition 3.1]{withoutlic},
    \begin{equation*}
        (\xi_\eta)_{V_x}=\Cm(x)\C(x)\xi_\eta.
    \end{equation*}
    Therefore
    \begin{equation*}
        \langle A_e(x)\xi_\eta,\xi_\eta\rangle_{\R^n}=\langle A_e(x)\Cm(x)\C(x)\xi_\eta,\Cm(x)\C(x)\xi_\eta\rangle_{\R^n}=\langle A(x)\C(x)\xi_\eta,\C(x)\xi_\eta\rangle_{\R^m},
    \end{equation*}
   whence
   \begin{equation*}
       f(x,\eta)\overset{\eqref{costonkeralbefabio}}{=}f(x,\C(x)\xi_\eta)=\langle A_e(x)\xi_\eta,\xi_\eta\rangle_{\R^n}=\langle A(x)\C(x)\xi_\eta,\C(x)\xi_\eta\rangle_{\R^m}.
   \end{equation*}
   The thesis follows combining the previous equation with \eqref{euclrapprgrad} and \Cref{MeySer}.
\end{proof}
The next result is preliminary to the convergence of momenta as stated in the forthcoming \Cref{thmmomentaconv}.
\begin{theorem}\label{teoremaperturbato3854868} Let $\Om$ be a bounded open set with Lipschitz boundary. Let $(X^h)_h\in \classuno\cup\classdue$. 
Let $(A^h)_h\subseteq E^h(\Omega;\lambda,\Lambda)$, $A\in E^\infty(\Omega;\lambda,\Lambda)$.
Let $(\Phi_h)_h \subseteq L^2(\Om;\R^n)$, $\Phi\in L^2(\Om;\R^n)$ be such that $\C^h\Phi_h\to\C\Phi\text{ strongly in $L^2(\Om,\rr^m)$.}$
Define $G^{\Phi_h}_h,G^\Phi:L^2(\Om)\times\mA\to[0,\infty]$ by
\begin{equation*}\label{perturbedlocalfunctionalsquadratic}
    G^{\Phi_h}_h(u,A)=\begin{cases}
    \int_A\langle A^h(X^hu+\C^h\Phi_h),X^hu+\C^h\Phi_h\rangle_{\R^m}{\rm d}x\quad&\text{if }u\in W^{1,2}_{X^h}(A),\\
    \infty\quad&\text{otherwise},
    \end{cases}
\end{equation*}
and
\begin{equation*}\label{perturbedlocalfunctionalsquadraticlimit}
    G^\Phi(u,A)=\begin{cases}
    \int_A\langle A(Xu+\C\Phi),(Xu+\C\Phi)\rangle_{\R^m}{\rm d}x\quad&\text{if }u\in W^{1,2}_{X}(A),\\
    \infty\quad&\text{otherwise}.
    \end{cases}
\end{equation*}
If \eqref{gammaconvstatement} holds for every $A\in\mA$ with Lipschitz boundary, then 
\begin{equation}\label{GammalimitG}
    G^\Phi(\cdot,A)=\Gamma(L^2)-\lim_{h\to\infty}G^{\Phi_h}_{h}(\cdot,A)\qquad\text{for every $A\in\mA$ with Lipschitz boundary}.
\end{equation}
\end{theorem}
\begin{proof}
Start from an arbitrary, not relabeled, subsequence. Let $F_h,F:L^2(\Om)\times\mA\to[0,\infty]$ be the quadratic forms given by $A^h$ and $A$.
Define
\[
g_h^{\Phi_h}(x,\eta)\coloneqq\,f_h(x,\eta+\C^h(x)\Phi_h(x))=\langle A^h(x)(\eta+\C^h(x)\Phi_h(x)),\eta+ \C^h(x)\Phi_h(x)\rangle_{\R^m}
\]
and 
\begin{equation*}
    \tilde g_h^{\Phi_h}(x,\eta)\coloneqq g_h^{\Phi_h}(x,\eta)+\lambda|\C^h(x)\Phi_h(x)|^2
\end{equation*}
for a.e. $x\in\Om$ and for any $\eta\in\R^m$.
Then, up to subsequences,
\begin{equation*}
    \begin{split}
        \tilde g_h^{\Phi_h}(x,\C^h(x)\xi)&=\langle A^h(x)\C^h(x)(\xi+\Phi_h(x)),\C^h(x)(\xi+\Phi_h(x))\rangle_{\R^m}+\lambda|\C^h(x)\Phi_h(x)|^2\\
        &\overset{\eqref{growthforHconv}}{\leq}\Lambda|\C^h(x)(\xi+\Phi_h(x))|^2+\lambda|\C^h(x)\Phi_h(x)|^2\\
        &\leq (2\Lambda+\lambda)|\C^h(x)\Phi_h(x)|^2+2\Lambda|\C^h(x)\xi|^2\\
        &\leq(2\Lambda+\lambda)|\tilde k(x)|+2\Lambda|\C^h(x)\xi|^2
    \end{split}
\end{equation*}
for a suitable $\tilde k\in L^1(\Om)$, where the last passage follows from \cite[Theorem 4.9]{MR2759829}. Similarly,
\begin{equation*}
      \tilde g_h^{\Phi_h}(x,\C^h(x)\xi)\overset{\eqref{growthforHconv}}{\geq}\lambda |\C^h(x)\xi|^2+2\lambda|\\C^h(x)\Phi_h(x)|^2+2\lambda\langle \C^h(x)\xi,\C^h(x)\Phi_h(x)\rangle_{\R^m}
      \geq \frac{\lambda}{2}|\C^h(x)\xi|^2
\end{equation*}
for a.e. $x\in\Om$ and for any $\xi\in\R^n$. Therefore
\begin{equation}\label{ghclass}
    \tilde g_h^{\Phi_h}\in\mathcal{I}^h_2\left((2\Lambda+\lambda)|\tilde k|,\frac{\lambda}{2},2\Lambda\right)
\end{equation}
for every $h\in\N$. Since $\C^h\Phi_h\to\C\Phi$ strongly in $L^2(\Om;\R^m)$,  \Cref{gammathmwithproof} implies that there exists $\hat G^\Phi:\,L^2(\Om)\times\mA\to[0,\infty]$ such that, up to subsequences, 
\begin{align*}
    \hat G^\Phi(\cdot,A)=\Gamma(L^2)-\lim_{h\to\infty}G^{\Phi_h}_{h}(\cdot,A)
\end{align*}
for every $A\in\mA$ with Lipschitz boundary, and that there exists a function $\hat g^\Phi$ such that
\begin{equation}\label{gPhiinI}
    \hat g^\Phi+\lambda|\C(x)\Phi(x)|^2\in\mathcal{I}^\infty_2\left((2\Lambda+\lambda)|\tilde k|,\frac{\lambda}{2},2\Lambda\right)
\end{equation}
and such that $\hat G^\Phi$ can be represented as
\[
\hat G^\Phi(u,A)\coloneqq
\displaystyle{\begin{cases}
\int_{A}\hat g^\Phi(x,Xu(x))\, dx&\text{ if }A\in\mA,\,u\in W^{1,2}_X(A),\\
\infty&\text{ otherwise}.
\end{cases}
}
\]
To conclude, we show that
\begin{equation}\label{thesis}
\hat G^\Phi(u,A)=G^\Phi(u,A)
\end{equation}
for each $A\in\mA$ with Lipschitz boundary and $u\in W^{1,2}_{X}(A)$.
We divide the proof of \eqref{thesis} into two steps.
\medskip

\noindent{\bf Step 1.} Let $\xi\in\R^n$. Set  $\Psi^h(x)=\Psi(x)=\xi$ for any $h\in\N$ and any $x\in\Om$. Let $G_h^{\Psi_h},G^\Psi$ and $\hat G^\Psi$ be the corresponding functionals. If $u_\xi(x)\coloneqq\langle\xi,x\rangle_{\R^n}$ for any $x\in\Rn$, then
\[
X^h u_\xi=\C^h\xi,\qquad Xu_\xi=\C\xi\qquad\text{and}\qquad G_h^{\Psi_h}(u,A)=\,F_h(u+\,u_\xi,A).
\]
Let $(u_h)_h$ be a recovery sequence for $u$ relative to $\hat G^\Psi$. Then
\begin{equation*}
    G^\Psi(u,A)= F(u+u_\xi,A)\leq\liminf_{h\to\infty}F_h(u_h+u_\xi,A)=\liminf_{h\to\infty} G^{\Psi_h}_h(u_h,A)\leq\hat G^\Psi(u,A).
\end{equation*}
Conversely, let $(v_h)_h$ be a recovery sequence for $u+u_\xi$ relative to $F$. As $v_h-u_\xi\to u$ strongly in $L^2(\Om)$,
\begin{equation*}
    \hat G^\Psi(u,A)\leq\liminf_{h\to\infty} G^\Psi_h(v_h-u_\xi,A)=\liminf_{h\to\infty} F_h(v_h,A)\leq F(u+u_\xi,A)=G^\Psi(u,A).
\end{equation*}
Therefore, \eqref{thesis} holds with $\Phi$ replaced by $\Psi$.
\medskip

\noindent\textbf{Step 2.}
Fix $A\in\mA$ with Lipschitz boundary, and
$u\in L^2(\Om)\cap W^{1,2}_X(A)$. Fix a ball $B\Subset A$ and
$\xi\in\R^n$. By \eqref{growthforHconv},
\begin{equation}\label{stimaconleradiciermvorgvm}
    \left|\sqrt{G_h^{\Phi_h}(v,B)}-\sqrt{G_h^\xi(v,B)}\right|
    \leq \sqrt{\Lambda}\,
    \|\C^h\Phi_h-\C^h\xi\|_{L^2(B;\R^m)},\qquad \text{for every }v\in L^2(\Om)\cap W^{1,2}_{X^h}(B).
\end{equation}
By \eqref{a0},
$\C^h\xi\to\C\xi$ in $L^2(B;\R^m)$, and the same holds for
$\C^h\Phi_h\to\C\Phi$ by assumption. Then, the right-hand side
of \eqref{stimaconleradiciermvorgvm} converges to
$\sqrt{\Lambda}\|\C(\Phi-\xi)\|_{L^2(B;\R^m)}$.
By Step~1, there exists $(v_h)_h$ such that $v_h\to u$ strongly in $L^2(\Om)$
and $G_h^\xi(v_h,B)\to G^\xi(u,B)$. Since $G^\xi(u,B)$ is
finite, $v_h$ belongs to $L^2(\Om)\cap W^{1,2}_{X^h}(B)$ for all sufficiently
large $h$. Then, \eqref{stimaconleradiciermvorgvm} yields
\begin{equation*}
    \sqrt{\hat G^\Phi(u,B)}
    \leq \liminf_{h\to\infty}\sqrt{G_h^{\Phi_h}(v_h,B)}
    \leq \sqrt{G^\xi(u,B)}
       +\sqrt{\Lambda}\|\C(\Phi-\xi)\|_{L^2(B;\R^m)}.
\end{equation*}
Conversely, choosing a recovery sequence for $\hat G^\Phi(\cdot,B)$
at $u$,
\begin{equation*}
    \sqrt{G^\xi(u,B)}
    \leq \sqrt{\hat G^\Phi(u,B)}
       +\sqrt{\Lambda}\|\C(\Phi-\xi)\|_{L^2(B;\R^m)}.
\end{equation*}
It is then simple to conclude that 
\begin{equation*}
    \left|\sqrt{\hat G^\Phi(u,B)}-\sqrt{G^\Phi(u,B)}\right|
    \leq 2\sqrt{\Lambda}\|\C(\Phi-\xi)\|_{L^2(B;\R^m)}.
\end{equation*}
Therefore, for a.e. $x\in A$ and for every $r\ges 0$ such that $B(x,r)\Subset A$, and choosing $\xi=\Phi(x)$,
\begin{equation*}
\begin{split}
    \left|
    \sqrt{\frac{\hat G^\Phi(u,B(x,r))}{|B(x,r)|}}
    -\sqrt{\frac{G^\Phi(u,B(x,r))}{|B(x,r)|}}
    \right|\leq 2\sqrt{\Lambda}\|\C\|_{L^\infty(A)}
    \left(\frac{1}{|B(x,r)|}\int_{B(x,r)}
    |\Phi(y)-\Phi(x)|^2\,{\rm d}y\right)^{1/2},
\end{split}
\end{equation*}
and the right-hand side tends to $0$ as $r\to 0$. 
Therefore, we infer that
\begin{equation*}
    \hat g^\Phi(x,Xu(x))
    =\langle A(x)(Xu(x)+\C(x)\Phi(x)),
                  Xu(x)+\C(x)\Phi(x)\rangle_{\R^m}
    \qquad\text{for a.e. }x\in A.
\end{equation*}
In particular, \eqref{thesis} follows integrating over $A$.
\end{proof}
We are in position to infer convergence of momenta in the following moving anisotropic sense.
\begin{theorem}\label{thmmomentaconv}
  Let $\Om$ be a bounded open set with Lipschitz boundary. Let $(X^h)_h\in\classuno\cup\classdue$. Let $(A^h)_h\subseteq E^h(\Omega;\lambda,\Lambda)$. Let $A^\infty\in E^\infty(\Omega;\lambda,\Lambda)$. 
 Assume that \eqref{gammaconvstatement} holds for every $A\in\mA$ with Lipschitz boundary. Then, up to subsequences,
    \begin{equation*}
    \int_A\langle A^\infty X u,\C\Phi\rangle_{\R^m}{\rm d}x  =\lim_{h\to\infty} \int_A\langle A^hX^h u_h,\C^h\Phi_h\rangle_{\R^m}{\rm d}x
    \end{equation*}
    for any $A\in\mA$ with Lipschitz boundary, any $u\in W^{1,2}_X(A)$, any recovery sequence $(u_h)_h$ relative to $F_\infty$, any $\Phi\in L^2(\Om;\R^n)$ and any $(\Phi_h)_h\subseteq L^2(\Om,\rr^n)$ such that $\C^h\Phi_h\to\C\Phi$ strongly in $L^2(\Om,\rr^m)$.
\end{theorem}
\begin{proof}
    For $h\in\N$, denote by $(F_h)_h$ and $F_\infty$ the quadratic forms related to $(A^h)_h$ and $A^\infty$.
    Fix $t\ges 0$. Denote by $G_h^{t\Phi_h}$ and $G^{t\Phi}$ the functionals given by \Cref{teoremaperturbato3854868}. By the properties of $(u_h)_h$ and \Cref{teoremaperturbato3854868}, 
    \begin{equation*}
        \frac{G^{t\Phi}(u,A)-F_\infty(u,A)}{t}\leq \liminf_{h\to\infty}\frac{G_h^{t\Phi_h}(u_h,A)-F_h(u_h,A)}{t}.
    \end{equation*}
    Notice that
    \begin{equation*}
        \begin{split}
             \frac{G^{t\Phi}(u,A)-F_\infty(u,A)}{t}&=2\int_A\langle A^\infty X u,\C\Phi\rangle_{\R^m}{\rm d}x+t\int_A \langle  A^\infty\C\Phi,\C\Phi\rangle_{\R^m}{\rm d}x
        \end{split}
    \end{equation*}
    and
     \begin{equation*}
        \begin{split}
             \frac{G_h^{t\Phi_h}(u_h,A)-F_h(u_h,A)}{t}&=2\int_A\langle A^h X^h u_h,\C^h\Phi_h\rangle_{\R^m}{\rm d}x+t\int_A \langle A^h\C^h\Phi_h,\C^h\Phi_h\rangle_{\R^m}{\rm d}x.
        \end{split}
    \end{equation*}
    Since, up to subsequences, 
    \begin{equation*}
        \left|\int_A \langle A^h\C^h\Phi_h,\C^h\Phi_h\rangle_{\R^m}{\rm d}x\right|\leq \Lambda\int_A|\C^h\Phi_h|^2{\rm d}x\leq \Lambda \left(|\Om|+\int_A|\C\Phi|^2{\rm d}x\right),
    \end{equation*}
    and similarly for the second term, we let $t\to 0$ and conclude that 
    \begin{equation}\label{auxiliariamomenti}
        \int_A\langle A^\infty X u,\C\Phi\rangle_{\R^m}{\rm d}x\leq\liminf_{h\to\infty}\int_A\langle A^h X^h u_h,\C^h\Phi_h\rangle_{\R^m}{\rm d}x.
    \end{equation}
    The thesis follows by evaluating \eqref{auxiliariamomenti} in $\Psi=-\Phi$ and $\Psi_h=-\Phi_h$.
\end{proof}
\color{black}
\section{\texorpdfstring{$H$}{H}-convergence for linear operators under moving anisotropies}\label{Section7}
Owing to the results proved in \Cref{Section6}, we consider the asymptotic behaviour of suitable anisotropic, symmetric, linear differential operators, extending previous results in the static case (cf. \cite{MR4230967,MR4536010,MR4504133}, and cf. \cite{caponi2024} for some recent non-local achievements).
Let $p=2$.
Let $(X^h)_h\in\classunoraff\cup\classdueraff$.
Let $(A^h)_h\subseteq E^h(\Omega;\lambda,\Lambda)$, for some fixed $0<\lambda\leq\Lambda$. Let $(\varphi_h)_h\subseteq W^{1,p}_{X^h}(\Om)$ and $\varphi\in W^{1,p}_{X}(\Om)$ be such that \eqref{quasichiusuradebolehpduetre} holds.
For $h\in\N$, let $\mathcal L^h$ be the linear anisotropic differential operator given formally by
\begin{equation*}\label{operL}
\mathcal L^h\coloneqq\,{\rm div}_{X^h}(A^h(x)X^h).
\end{equation*}
For any $f\in L^2(\Omega)$ and any $\mu \geq 0$, 
we consider anisotropic elliptic problems of the form
\begin{equation}\label{phformal}\tag{$P_h$}
\begin{cases}
\mu u_h-\mathcal{L}^h(u_h)=f&\text{in }\Omega,\\
u_h\in W^{1,2}_{X^h,\varphi_h}(\Omega).
\end{cases}
\end{equation}
Accordingly, we say that $u_h\in W^{1,2}_{X^h,\varphi_h}(\Om)$ is a \emph{weak solution} to \eqref{phformal} if
\begin{equation*}
    \mu\int_\Om u_hv\,{\rm d}x+\int_{\Omega}\langle A^hX^hu_h,X^hv\rangle_{R^m}\,{\rm d}x=\int_\Omega fv\,{\rm d}x
\end{equation*}
for any $v\in W^{1,2}_{X_h,0}(\Om)$. If each $X^h$ belongs to some $\mathcal{M}_h=\mathcal{M}_h(\Om^h_0,C^h_d,c_h,C_h)$, then \Cref{poincaprelprop}, \Cref{relkon}, \Cref{acosaloapplichiamo} and a standard application of the direct methods as e.g. in \cite[Corollary 2.6]{Capogna2024} imply the existence of a unique minimizer $u_h\in W^{1,2}_{X^h,\varphi_h}(\Om)$ of the functional
\begin{equation*}
    u\mapsto \frac{\mu}{2}\int_\Om|u|^2\,{\rm d}x+\frac{1}{2}\int_\Om\langle A^hX^hu,X^hu\rangle_{R^m}\,{\rm d}x-\int_\Omega fu\,{\rm d}x,
\end{equation*}
which in turn is the unique weak solution to \eqref{phformal}. We stress that, as we are assuming $(X^h)_h\in\classunoraff\cup\classdueraff$, all our previous considerations hold replacing $X^h$ with the limit $X$-gradient $X$.
Under the above assumption, we provide the definition of $H$-convergence for sequences of operators $(\mathcal L^h)_h\subseteq  \mathcal{E}^h(\Omega;\lambda,\Lambda)$.
\begin{definition}\label{hconvdef}
    Let $(\mathcal L^h)_h\subseteq  \mathcal{E}^h(\Omega;\lambda,\Lambda)$. We say that $(\mathcal L^h)_h\subseteq  \mathcal{E}^h(\Omega;\lambda,\Lambda)$ \emph{$H$-converges to} $\mathcal L^\infty\in \mathcal{E}^\infty(\Omega;\lambda,\Lambda)$ if the following holds. For any $f\in L^2(\Om)$, any $\mu\geq 0$, any $(\varphi_h)_h\subseteq W^{1,p}_{X^h}(\Om)$ and any $\varphi\in W^{1,p}_{X}(\Om)$ such that \eqref{quasichiusuradebolehpduetre} holds, let $(u_h)_h\subseteq W^{1,2}_{X_h,\varphi_h}(\Om)$ and $u_\infty\in W^{1,2}_{X,\varphi}(\Om)$ be, respectively, the unique weak solutions to \eqref{phformal} and the corresponding limit problem. Then
    \begin{equation*}\label{convsol}
        u_h\to u_\infty\text{ strongly in }L^2(\Om)\qquad\text{(convergence of solutions)}
    \end{equation*}
    and 
     \begin{equation*}
    \int_\Om\langle A^hX^h u_h,\C^h\Phi_h\rangle_{\R^m}{\rm d}x\to  \int_\Om\langle A^\infty X u_\infty,\C\Phi\rangle_{\R^m}{\rm d}x \qquad\text{(convergence of momenta)}
    \end{equation*}
    for any $\Phi\in L^2(\Om;\R^n)$ and any $(\Phi_h)_h\subseteq L^2(\Om,\rr^n)$ such that $\C^h\Phi_h\to\C\Phi$ strongly in $L^2(\Om,\rr^m)$.
\end{definition}
Owing to the results presented in \Cref{Section6}, the following $H$-compactness theorem holds.
\begin{theorem}\label{teoremahcompattezza455y65hy}
Let $\Om$ be a bounded open set with Lipschitz boundary. Let $(X^h)_h\in \classunoraff\cup\classdueraff$. Assume that, for any $h\in \N$, $X^h$ belongs to some $\mathcal M_h$ as in \Cref{deficlassminipoinc}. Let $(\mathcal L^h)_h\subseteq  \mathcal{E}^h(\Omega;\lambda,\Lambda)$. Then, up to a subsequence, there exists $\mathcal L^\infty\in  \mathcal{E}^\infty(\Omega;\lambda,\Lambda)$ such that $\mathcal L^h$ $H$-converges to $\mathcal L^\infty$.
\end{theorem}
\begin{proof}
   It follows by \Cref{mainthmdirichletbc}, \Cref{minimaminimizerthm}, \Cref{acosaloapplichiamo}, \Cref{compactnessL2perHconvergenza} and \Cref{thmmomentaconv}.
\end{proof}
\section*{Declarations}
\footnotesize{
%
\noindent{\it \textbf{Funding information.}} 
The authors are member of \emph{Istituto Nazionale di Alta Matematica (INdAM--GNAMPA)}, and were supported by INdAM--GNAMPA 2023 Project \textit{Equazioni differenziali alle derivate parziali di tipo misto o dipendenti da campi di vettori} (CUP\_ E53\-C22\-001\-930\-001).  
S. V.\ was supported  by INdAM-GNAMPA 2025 Project \emph{Structure of sub-Riemannian hypersurfaces in Heisenberg groups} (CUP\_ES324001950001), and received funding by the European Union: the European
Research Council (ERC), through StG “ANGEVA”, project number: 101076411. Views
and opinions expressed are however those of the authors only and do not necessarily reflect
those of the European Union or the European Research Council. Neither the European
Union nor the granting authority can be held responsible for them.
A.M.\ was supported by the Spanish State Research Agency, through the Severo Ochoa and María de Maeztu Program for Centers and Units of Excellence in R\&D (CEX2020-001084-M), by MCIN/AEI/10.13039/501100011033 (PID2021-123903NB-I00), by Generalitat de Catalunya (2021-SGR-00087). 
A.M.\ was supported by the INdAM-GNAMPA 2024 Project \emph{Pairing and div-curl lemma: extensions to weakly differentiable vector fields and nonlocal differentiation} (CUP\_E53C23001670001).
A.M. and F.P.\ were supported by the INdAM - GNAMPA 2025 Project \emph{Metodi variazionali per problemi dipendenti da operatori frazionari isotropi e anisotropi} (CUP\_E5324001950001).
This research was supported by the \emph{Centre de Recerca Matemàtica} under the International Programme for Research in Groups.
\smallskip

\noindent{\it \textbf{Conflict of interest.}} The authors disclose financial interests or conflicts of interest related to the subject matter.
\smallskip

\noindent{\it \textbf{Data availability statement.}} Data sharing not applicable as no datasets were generated or analysed during this study.}
\bibliographystyle{abbrv}
\begingroup
\let\originalbibitem\bibitem
\def\bibitem#1{\par
  \ifboolexpr{test {\ifstrequal{#1}{BiagiBonfiglioliBramanti2022}} or
              test {\ifstrequal{#1}{Calderon1965Commutators}} or
              test {\ifstrequal{#1}{DuitsFranken2010}}}
    {\color{purple}}{\normalcolor}%
  \originalbibitem{#1}}
\bibliography{biblio}
\endgroup
\end{document}